\documentclass[11pt,twoside,reqno]{amsart}

\usepackage[a4paper,margin=1in]{geometry}
\usepackage{amsmath,amssymb,amsthm,mathtools,bm,mathrsfs}
\usepackage{graphicx}
\usepackage{enumitem}
\usepackage{fancyhdr}
\usepackage[colorlinks=true,linkcolor=blue,citecolor=blue,urlcolor=blue]{hyperref}

\fancypagestyle{plain}{%
  \fancyhf{}%
  \fancyfoot[C]{\thepage}%
}

\numberwithin{equation}{section}

\theoremstyle{plain}
\newtheorem{proposition}{Proposition}[section]
\newtheorem{theorem}[proposition]{Theorem}
\newtheorem{lemma}[proposition]{Lemma}
\newtheorem{corollary}[proposition]{Corollary}
\theoremstyle{remark}
\newtheorem{remark}[proposition]{Remark}

\newcommand{\T}{\mathbb T}
\newcommand{\R}{\mathbb R}
\newcommand{\dd}{\,\mathrm d}
\newcommand{\Div}{\operatorname{div}}
\newcommand{\grad}{\operatorname{grad}}
\newcommand{\PP}{\mathbb P}
\newcommand{\Kcal}{\mathcal K}
\newcommand{\Om}{\boldsymbol\Omega}

\title[Floquet instability near a stationary Euler flow]
{Floquet instability near a stationary Euler flow in 3D}
\author[Mimi Dai]{Mimi Dai}
\address{Department of Mathematics, Statistics and Computer Science,
University of Illinois at Chicago, Chicago, IL 60607, USA}
\email{mdai@uic.edu}
\date{}

\begin{document}
\begin{abstract}
We study the three-dimensional (3D) incompressible Euler equation linearized about a smooth compactly supported stationary flow of Gavrilov type.  Baldi's volume-preserving action--angle chart allows a reduction of the high-frequency amplitude equation on invariant tori.  Using Baldi’s chart, we identify a critical torus ensured by the compact cutoff. We show that, for a fixed axisymmetric covector with sufficiently large normal component, the reduced trace-free periodic cocycle on this torus is hyperbolic and has a positive Floquet exponent.  A localized geometric-optics construction then lifts this Floquet instability to the linearized Euler equation, yielding exponential essential-norm growth in space $H^s$, $s\geq 0$. 
\end{abstract}

\maketitle

\begingroup
\setlength{\parskip}{0pt}
\tableofcontents
\endgroup

\section{Introduction and main results}
\label{sec-intro}

Consider the incompressible Euler equations on $\R^3$
\begin{equation}
  \partial_t u+(u\cdot\nabla) u+\nabla p=0,
  \qquad \Div u=0
  \label{eq:euler}
\end{equation}
where $u$ denotes the velocity field and $p$ the scalar pressure function. 
Gavrilov \cite{Gavrilov2019} constructed a nontrivial smooth compactly supported steady solution of the 3D Euler equations, and Constantin, La and Vicol \cite{ConstantinLaVicol2019} gave a localization framework for the same Grad--Shafranov structure.  Baldi \cite{Baldi2024} later showed that the particle dynamics on small active tori admits a volume-preserving
action--angle representation.  The purpose of this paper is to study the stability of the linearized Euler equation around such compactly supported steady states.

\subsection{The localized equilibrium}
\label{sec-localized}

Gavrilov \cite{Gavrilov2019} proved the existence of a nonzero velocity
 $ U\in C_0^\infty(\R^3;\R^3)$
that solves the stationary 3D Euler equations,
\begin{equation*}
 (U\cdot \nabla) U+\nabla P=0,
  \qquad \Div U=0
\end{equation*}
with steady pressure $P$. Moreover, the support of $U$ can be placed in an arbitrarily small neighborhood of a circle.  The solution is axisymmetric and has nonzero swirl. It is known that a smooth compactly supported axisymmetric steady flow without swirl must vanish.

The construction begins with a local analytic solution near a circle
$\{r=R,z=0\}$ in cylindrical coordinates.  Gavrilov builds a scalar function
$a=a(r,z)$ with a strict local minimum on that circle and defines an uncut
local pair $(U_0,P_0)$ by
\begin{equation*}
  P_0=\frac{R^4}{4}a,
  \qquad
  U_0=\frac1r\bigl((P_0)_z e_r-(P_0)_r e_z+b(a)e_\varphi\bigr),
\end{equation*}
where $b(a)$ is chosen through an auxiliary analytic ODE so that the steady
Euler equations hold.  The pair $(U_0, P_0)$ satisfies
\begin{equation}
  U_0\cdot\nabla P_0=0,
  \qquad
  |U_0|^2=3P_0.
  \label{eq:Gavrilov-first-integrals}
\end{equation}
For smooth $\chi$, denote
\begin{equation}
  U=\chi(P_0)U_0,
  \qquad
  P=\mathcal W(P_0),
  \qquad \mathcal W'(\cdot)=\chi(\cdot)^2.
  \label{eq:localized-pressure-primitive}
\end{equation}
Then
\begin{equation}
  \nabla P=\chi(P_0)^2\nabla P_0
  \label{eq:Gavrilov-localization}
\end{equation}
and $(U,P)$ is again a steady solution of the incompressible Euler equations.  The uncut local
formula is not smooth on the central circle due to the swirl factor.
Gavrilov chooses $\chi$ supported in a thin interval of positive $P_0$-values,
so the cutoff removes both that circle and the exterior of the local
coordinate neighborhood.  The final velocity is smooth and supported in a
thin toroidal shell.  Thus the compact support is obtained by cutting along
invariant $P_0$-level surfaces. 

Constantin, La, and Vicol \cite{ConstantinLaVicol2019} recast Gavrilov's insight as a systematic
Grad--Shafranov construction.  For the
axisymmetric uncut ansatz
\begin{equation*}
  U_0=\frac1r\bigl(\psi_z e_r-\psi_r e_z+F(\psi)e_\varphi\bigr),
\end{equation*}
the stationary Euler equations reduce to the Grad--Shafranov equation
\begin{equation}
  -(\partial_r^2-\frac1r\partial_r+\partial_z^2)\psi
  =\partial_\psi\!\left(\frac{F(\psi)^2}{2}+r^2\Pi(\psi)\right).
  \label{eq:CLV-GS-equation}
\end{equation}
A key point of \cite{ConstantinLaVicol2019} is that the \emph{localizability constraint}
\begin{equation}
  \frac{|U_0|^2}{2}=A(\psi)
  \label{eq:CLV-localizability}
\end{equation}
makes the pressure $P_0=-\Pi(\psi)-A(\psi)$ a function of $\psi$.  Since
$U_0\cdot\nabla\psi=0$, it shows $U_0\cdot\nabla P_0=0$ and hence allows the
same cutoff in the pressure $P_0$ as in
\eqref{eq:Gavrilov-localization}.  The coupled
system of \eqref{eq:CLV-GS-equation} and
\eqref{eq:CLV-localizability} is overdetermined.  They solve it locally by a
hodograph transformation, reducing compatibility to analytic ODEs for the
functions $A$, $\Pi$, and $F^2$.  The stream function has a strict local
minimum at a toroidal core.  Before cutoff the associated velocity is only
H\"older at the core and its vorticity is bounded; a cutoff in $P_0$ supported
away from the core produces the required $C_0^\infty$ steady Euler flow in a
toroidal shell.

The same localizable building block can be translated, rotated, and rescaled,
and disjoint copies can be superposed because their velocity and pressure
supports do not interact.  Constantin, La and Vicol apply this observation to
construct stationary solutions with prescribed H\"older thresholds, including
locally smooth $L^2\cap C^{1/3}$ examples that fail to lie in any
$C^\beta$, $\beta>1/3$, and can have arbitrarily large
$\|\,|\nabla U||U|^2\,\|_{L^\infty}$ while their local energy dissipation
vanishes.  

More recently, Peralta-Salas and Slobodeanu
\cite[Theorem~1.1]{PeraltaSalasSlobodeanu2026} proved axisymmetry and
convex toroidal geometry for localizable steady Euler flows that are
analytic in a smooth bounded domain and $C^1$ up to its boundary.
Their result assumes that the velocity is tangent to the boundary and
that the Bernoulli function is constant on each boundary component, with
nonvanishing gradient there.

\subsection{Baldi's description of Gavrilov’s particle dynamics}
\label{sec-Baldi}

Baldi's main result \cite{Baldi2024} is an explicit description for
the Lagrangian trajectories of Gavrilov's localized steady Euler flows near
the central circle.  After normalizing the radius of that circle to one, Baldi
constructs an analytic diffeomorphism
\begin{equation*}
  \Phi:\T^2\times(0,I^*)\longrightarrow S^*
\end{equation*}
that associates the particle equation $\dot x=U(x)$ to
\begin{equation}
  \dot\sigma=\Omega_1(I),\qquad
  \dot\beta=\Omega_2(I),\qquad
  \dot I=0.
  \label{eq:Baldi-integrable-system}
\end{equation}
Consequently, a particle with initial data
$x_0=\Phi(\sigma_0,\beta_0,I_0)$ follows the trajectory
\begin{equation*}
  x(t)=\Phi\bigl(\sigma_0+\Omega_1(I_0)t,
  \beta_0+\Omega_2(I_0)t,I_0\bigr).
\end{equation*}
The pair $(\sigma,I)$ is an angle--action pair for the one-degree-of-freedom
Hamiltonian subsystem describing motion in a meridional plane.  The second
angle $\beta$ is obtained from the physical azimuthal angle by subtracting a
periodic libration.  

The action variable has both a dynamical and a geometric meaning.  Baldi
proves that it is a reparametrization of the uncut analytic pressure $P_0$:
\begin{equation}
  P_0\bigl(\Phi(\sigma,\beta,I)\bigr)=\Kcal(I).
  \label{eq:Baldi-pressure-action}
\end{equation}
The pressure $P$ of the localized equilibrium is instead
\begin{equation}
  P\bigl(\Phi(\sigma,\beta,I)\bigr)=\mathcal W(\Kcal(I))
  \qquad \mbox{with} \quad \mathcal W'(\cdot)=\chi(\cdot)^2.
  \label{eq:localized-pressure-action}
\end{equation}
Thus each surface $I=\mathrm{constant}$ is an invariant $P_0$-level torus;
on the active set, where $\chi(\Kcal(I))\ne0$, it is also a level torus of
the localized pressure $P$.  Near
the central circle, the coordinate map has the form
\begin{equation}\label{eq:Phi-explicit}
  \Phi(\sigma,\beta,I)
  =\begin{pmatrix}
  \rho(\sigma,I)\cos\bigl(\beta+\gamma(\sigma,I)\bigr)\\
  \rho(\sigma,I)\sin\bigl(\beta+\gamma(\sigma,I)\bigr)\\
  \zeta(\sigma,I)
  \end{pmatrix},
\end{equation}
where
\begin{equation*}
  \rho(\sigma,I)=1+\sqrt{2I}\sin\sigma+O(I),\qquad
  \gamma(\sigma,I)=O(I),\qquad
  \zeta(\sigma,I)=\sqrt{2I}\cos\sigma+O(I).
\end{equation*}
The expansions show that the $P_0$-levels are nearly circular tori.  The
map extends continuously to $I=0$, where the two-torus collapses to the
central circle, and it admits a convergent expansion in powers of $\sqrt I$.
Figure~\ref{fig:Baldi-canonical-meridional-coordinates} illustrates the two
geometric stages of this construction.  
Revolving one such $P_0$-level loop about the symmetry axis produces the invariant torus
$I=\mathrm{constant}$.

\begin{figure}[t]
  \centering
  \includegraphics[width=\textwidth]{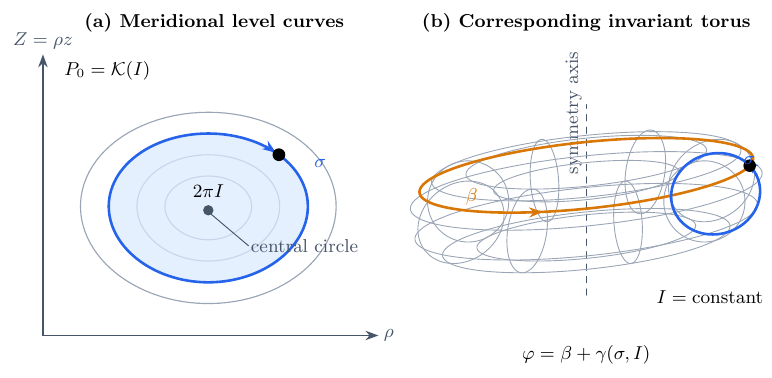}
  \caption{Schematic illustration of Baldi's canonical coordinates.
  Left: in the meridional plane $(\rho,Z)$, with $Z=\rho z$, a fixed
  action selects a closed $P_0$-level curve enclosing area $2\pi I$;
  $\sigma$ is the meridional angle.  Right: rotating this curve about the
  symmetry axis produces an invariant torus.  The action $I$ is fixed,
  $\beta$ is the straightened azimuthal angle, and the physical azimuth is
  $\varphi=\beta+\gamma(\sigma,I)$.  The black dot represents the same fluid
  particle in the two figures.  
  }
  \label{fig:Baldi-canonical-meridional-coordinates}
\end{figure}

Moreover,
\begin{equation*}
  \det D\Phi=1,
  \qquad
  \operatorname{vol}\{x:P_0(x)<\Kcal(I)\}=(2\pi)^2I.
\end{equation*}

Baldi also calculates the two angular frequencies
\begin{equation}\label{eq:Omega2}
  \Omega_1(I)=\chi(\Kcal(I))\Kcal'(I),
  \qquad
  \Omega_2(I)=\sqrt I\,\mathcal R(I)\Omega_1(I)
\end{equation}
with convergent expansions
\begin{equation}
  \Kcal(I)=I+\frac{1065}{1024}I^3+O(I^4),
  \qquad
  \mathcal R(I)=1+\frac74I+O(I^2).
  \label{eq:Baldi-expansion-summary}
\end{equation}
In particular, the rotation number
\begin{equation*}
  \frac{\Omega_2(I)}{\Omega_1(I)}
  =\sqrt I\,\mathcal R(I)
\end{equation*}
is increasing for sufficiently small $I>0$.  The functions
$\Phi$, $\Kcal$, and $\mathcal R$ do not depend on the
choice of cutoff.

\subsection{Main results}
\label{sec-main-results}
Baldi's result provides the volume-preserving integrable coordinates
in which we analyze the linearized Euler dynamics.
We restrict attention to a closed toroidal band strictly contained in the
active region of the steady velocity.  In particular, the band stays away
from the infinitely flat boundaries introduced by the compactly supported
cutoff.

Let $\chi_*$ and $\iota$ denote the cutoff parameters introduced in
Section~\ref{sec-cutoff}.  Below is the main result.

\begin{theorem}\label{thm:critical-torus-exponential-essential-instability}
Fix $\chi_*>0$.
For all sufficiently small $\iota$, let $U_\iota$ be the
localized Gavrilov equilibrium associated with
\eqref{eq:explicit-pressure-cutoff}, and let $G_\iota(t)$ be the
linearized Euler evolution operator.  There are a critical torus $I_{\rm c}$, an
integer $n\in\mathbb Z\setminus\{0\}$, a sufficiently large $\ell>0$, and a
constant $\sigma_\ell>0$ associated with the axisymmetric covector
$(n,0,\ell)$ such that, for every $s\ge0$,
\begin{equation}
  \liminf_{t\to\infty}\frac1t
  \log\left\|
    G_\iota(t)|_{H^s_{\rm axi}}
  \right\|_{{\rm ess},H^s}
  \ge\sigma_\ell>0.
  \label{eq:critical-torus-exact-exponential-essential-growth}
\end{equation}
\end{theorem}
Here $H^s_{\rm axi}$ is the closed axisymmetric
divergence-free subspace defined in Section~\ref{sec-lift}.
As an immediate consequence, for every $0<\delta<\sigma_\ell$ and all sufficiently large
observation times $T$, there is a real smooth axisymmetric divergence-free
field $v_{0,T}$ with $\|v_{0,T}\|_{H^s}=1$ such that
\begin{equation}
  \|G_\iota(T)v_{0,T}\|_{H^s}
  \ge e^{(\sigma_\ell-\delta)T}.
  \label{eq:critical-torus-observation-time-real-packet}
\end{equation}
Thus the testing packet may depend on the observation time.  Nevertheless, we can show a stronger fixed-data statement by the uniform boundedness principle
and Baire category.

Denote
\begin{equation*}
  H^s_{\rm div}(\R^3;\R^3)
  :=\{v\in H^s(\R^3;\R^3):\Div v=0\}, \quad s\geq0.
\end{equation*}

\begin{corollary}\label{cor:generic-fixed-Sobolev-datum}
Let $\sigma_\ell>0$ be the constant from
Theorem~\ref{thm:critical-torus-exponential-essential-instability}.  For every
$s\ge0$, there is a residual subset
$\mathcal R_s$ of the real Banach space
$H^s_{\rm div}(\R^3;\R^3)$ such that every nonzero
$v_0\in\mathcal R_s$ satisfies
\begin{equation}
  \limsup_{t\to\infty}
  \frac1t\log\bigl(1+\|G_\iota(t)v_0\|_{H^s}\bigr)
  \ge\sigma_\ell.
  \label{eq:fixed-datum-subsequential-growth}
\end{equation}
The same conclusion holds for a residual subset of the real axisymmetric
subspace. In particular, $v_0$ may be normalized in $H^s$, and
$\sup_{t\ge0}\|G_\iota(t)v_0\|_{H^s}=\infty$.
\end{corollary}

\subsection{Motivation of the main results}
\label{sec-motivation}
The short-wave approach and geometric-optics theory of
hydrodynamic instability was developed in the early 1990s independently by Friedlander--Vishik  \cite{FriedlanderVishik1991,FriedlanderVishik1992,VishikFriedlander1993} and Lifschitz--Hameiri \cite{LifschitzHameiri1991, LifschitzHameiri1993}.  Lifschitz and Hameiri derived local stability equations for rapidly oscillating
perturbations of a general fluid flow, while Friedlander and Vishik obtained geometric
lower bounds and instability criteria for 3D steady Euler
flows.  In particular, Friedlander and Vishik
\cite{FriedlanderVishik1992} obtained a sufficient instability criterion
for axisymmetric integrable steady flows in terms of streamline curvature,
geodesic torsion, and helicity.  Starting from a WKB ansatz
\begin{equation*}
  v^\varepsilon(x,t)
  \sim a(x,t)e^{iS(x,t)/\varepsilon},
  \qquad \varepsilon\downarrow0,
\end{equation*}
one is led along a particle trajectory to the bicharacteristic-amplitude
system for the position, phase covector, and velocity polarization.  Thus
the leading short-wave dynamics is a finite-dimensional linear
cocycle over the Lagrangian flow.  Lifschitz and Hameiri subsequently applied
this method to axisymmetric vortex rings with swirl and identified both exponential and
algebraic localized instabilities \cite{LifschitzHameiri1993}. In particular, they showed that unbounded transport solutions fall into exponential or algebraic growth regimes. 

Different from \cite{LifschitzHameiri1993}, the present paper addresses a specific type of steady velocities in $C_c^\infty(\R^3;\R^3)$ -- the localized Gavrilov family. We first identify an instability mechanism specific to compact localization. The cutoff makes the meridional frequency positive in the active shell and zero at its flat boundary; consequently it forces an interior critical torus. This torus is not a stagnation set: its particles continue to move, while the selected axisymmetric mode has zero projected twist. We compute the complete reduced $2\times2$ periodic generator on this torus and prove that a sufficiently large fixed normal covector makes its Floquet monodromy hyperbolic.  Thus the unstable ray is produced by the geometry of the explicit equilibrium rather than imposed as an abstract
hypothesis.

Second, we perform a quantitative realization of that ray in the linearized Euler evolution.  Combining the explicit cocycle
calculation with the classical transfer theory yields exponential essential-norm growth on every $H^s$, already after restriction to the real
axisymmetric divergence-free subspace.  The uniform boundedness principle and Baire category then convert this operator growth into exponential subsequential growth for generic fixed data.

More specifically, for the localized Gavrilov family considered here, our mechanism requires neither an unstable eigenvalue nor positive Lyapunov stretching of the underlying particle flow; instead, the instability mechanism lies in the Floquet hyperbolicity for a covector with sufficiently large normal component.
A companion paper develops the corresponding nonlinear instability analysis based on this Floquet mechanism.

\subsection{Relation to prior work}
The connection from the amplitude dynamics to the spectrum and
growth of the exact linearized Euler evolution was placed on an
operator theory by Vishik \cite{Vishik1996}.  In particular, the essential spectral
radius is governed by the maximal Lyapunov exponent of the
bicharacteristic-amplitude cocycle.  Shvydkoy and Vishik \cite{ShvydkoyVishik2004}
then related individual Lyapunov--Oseledets exponents of the three-dimensional
Euler cocycle to the spectrum of the evolution group.  Shvydkoy  \cite{Shvydkoy2006} developed a general pseudodifferential
framework for advective equations, proving a corresponding transfer from
the dynamical spectrum of the amplitude cocycle to the essential spectrum
of the PDE; later work gave a sharper spectral picture
for three-dimensional Euler \cite{Shvydkoy2010}.  At the nonlinear level,
Friedlander, Strauss, and Vishik \cite{FriedlanderStraussVishik1997} established an abstract mechanism by which
appropriate linear instability implies nonlinear instability for ideal
fluids.
Friedlander and Shnirelman \cite{FriedlanderShnirelman2001} subsequently surveyed these developments from the
broader viewpoint of instability of steady ideal-fluid flows, emphasizing
unstable essential dynamics, the distinction between slow and fast
mechanisms, and the dependence of stability on the ambient norm.  

The mechanism exploited here is complementary
to the classical sufficient criterion based on positive Lyapunov stretching
of the particle flow.  On the relevant critical torus the integrable base
trajectory has zero classical Lyapunov exponent, whereas the constrained
Euler amplitude cocycle above that trajectory has a positive Floquet
exponent.  It is this exponent, rather than exponential separation of
nearby fluid particles or an isolated unstable eigenvalue, that yields the
exponential essential-norm growth in the main results.
The possibility of positive fluid-amplitude growth with zero classical
particle-flow Lyapunov exponents already appears in the integrable-flow
instability theory of Friedlander and Vishik
\cite{FriedlanderVishik1992}, as discussed in
\cite[Section~2.1]{FriedlanderShnirelman2001}.  The contribution here is a quantitative Floquet-instability theorem for the localized Gavrilov family, obtained by identifying a cutoff-generated critical torus and proving hyperbolicity.

The organization of the paper is as follows.  Section \ref{sec-geometry} develops the
action--angle geometry of the localized equilibrium, and records the
linearized velocity equation in these coordinates. Section \ref{sec-cutoff} introduces
the angular modes and the localized frequency profile.  Section \ref{sec-Floquet-growth} derives
the trace-free amplitude cocycle, computes its first averaged correction,
and proves Floquet hyperbolicity on a cutoff-generated critical torus.
Section \ref{sec-lift} lifts this growth to the linearized Euler evolution and establishs
the fixed-datum consequence.  Appendix~A collects the formulas from Baldi's
construction used in the small-action calculations.

\section{Geometry of the action--angle chart and the linearized operator}
\label{sec-geometry}

Recall Baldi's coordinate map $\Phi$ has the form \eqref{eq:Phi-explicit}.
Let
\begin{equation*}
  e_a=\partial_a\Phi,
  \qquad
  g_{ab}=e_a\cdot e_b,
  \qquad
  (g^{ab})=(g_{ab})^{-1}.
\end{equation*}
Since $\det D\Phi=1$, we have
\begin{equation}
  \det g=(\det D\Phi)^2=1.
  \label{eq:det-g}
\end{equation}
Consequently, the Euclidean volume form pulls back to
\begin{equation*}
  \dd x=\dd\sigma\,\dd\beta\,\dd I,
\end{equation*}
and a vector field $v=v^a e_a$ satisfies
\begin{equation*}
  \Div_g v=\partial_a v^a.
\end{equation*}
The gradient and scalar Laplace--Beltrami operator are
\begin{equation*}
  (\grad_g q)^a=g^{ab}\partial_bq,
  \qquad
  \Delta_gq=\partial_a\bigl(g^{ab}\partial_bq\bigr).
\end{equation*}
It follows that
\begin{equation}
  g=\dd\rho^2+\dd\zeta^2+
  \rho^2(\dd\beta+\dd\gamma)^2
  \label{eq:metric-compact}
\end{equation}
with components
\begin{equation}  \label{eq:metric-components-1}
\begin{split}
  g_{\sigma\sigma}
    &=\rho_\sigma^2+\zeta_\sigma^2+
      \rho^2\gamma_\sigma^2,
  \quad g_{\sigma\beta}=\rho^2\gamma_\sigma,
  \quad g_{\beta\beta}=\rho^2,\\
  g_{\beta I}&=\rho^2\gamma_I, \quad
  g_{\sigma I}
    =\rho_\sigma\rho_I+\zeta_\sigma\zeta_I+
      \rho^2\gamma_\sigma\gamma_I,
  \quad g_{II}=\rho_I^2+\zeta_I^2+\rho^2\gamma_I^2.
\end{split}
\end{equation}
The Christoffel symbols are
\begin{equation*}
  \Gamma^a_{bc}
  =\frac12g^{ad}
  \left(
  \partial_bg_{cd}+\partial_cg_{bd}-\partial_dg_{bc}
  \right).
\end{equation*}

\subsection{The first small-action expansion of the geometry}

We now carry out the first geometric expansion of the coordinate map. Denote $r:=\sqrt{2I}$.

\begin{proposition}\label{prop:first-geometric-expansion}
As $I\downarrow0$, Baldi's coordinate functions in
\eqref{eq:Phi-explicit} satisfy
  \begin{align}
  \rho(\sigma,I)
    &=1+r\sin\sigma+r^2 A(\sigma)+O(r^3),
  \label{eq:rho-first-expansion}\\
  \zeta(\sigma,I)
    &=r\cos\sigma+r^2 B(\sigma)+O(r^3),
  \label{eq:zeta-first-expansion}\\
  \gamma(\sigma,I)
    &=r^2 C(\sigma)+O(r^3)
  \label{eq:gamma-first-expansion}
\end{align}
uniformly in $\sigma$, with
\begin{equation}
  A=\frac{5\cos\sigma-2\cos^2\sigma-3}{4},
  \qquad
  B=\frac{\sin\sigma(2\cos\sigma-5)}{4},
  \qquad
  C=\sqrt2\,(\cos\sigma-1).
  \label{eq:ABC-first-expansion}
\end{equation}
Consequently, we have in the coordinate ordering $(\sigma,\beta,I)$, 
\begin{equation}
  (g_{ab})=
  \begin{pmatrix}
    r^2+r^3\sin\sigma 
      &-\sqrt2\,r^2\sin\sigma
      &\dfrac{5-8\cos\sigma}{4}r \\[2mm] 
    -\sqrt2\,r^2\sin\sigma 
      &1+2r\sin\sigma
        +\dfrac{5\cos\sigma-4\cos^2\sigma-1}{2}r^2 
      &2\sqrt2(\cos\sigma-1)\\[2mm] 
    \dfrac{5-8\cos\sigma}{4}r 
      &2\sqrt2(\cos\sigma-1)
      &r^{-2}-3r^{-1}\sin\sigma 
  \end{pmatrix}
  + (g_{ab})_R,
  \label{eq:metric-first-expansion}
\end{equation}
and
\begin{equation}
  (g^{ab})=
  \begin{pmatrix}
    r^{-2}-r^{-1}\sin\sigma 
      &\sqrt2\,\sin\sigma 
      &-\dfrac{5-8\cos\sigma}{4}r\\[2mm]
    \sqrt2\,\sin\sigma
      &1-2r\sin\sigma
      &2\sqrt2(1-\cos\sigma)r^2\\[2mm] 
    -\dfrac{5-8\cos\sigma}{4}r
      &2\sqrt2(1-\cos\sigma)r^2
      &r^2+3r^3\sin\sigma
  \end{pmatrix}
  + (g^{ab})_R
  \label{eq:inverse-metric-first-expansion}
\end{equation}
with the lower order terms 
\begin{equation*}
  (g_{ab})_R=
  \begin{pmatrix}
    O(r^4) &O(r^3) &O(r^2)\\
    O(r^3) &O(r^3) &O(r)\\
   O(r^2) &O(r) &O(1)
  \end{pmatrix}, \quad
   (g^{ab})_R=
  \begin{pmatrix}
   O(1)
      &O(r)
      &O(r^2)\\
  O(r)
      &O(r^2)
      &O(r^3)\\
   O(r^2)
      &O(r^3)
      &O(r^4)
  \end{pmatrix}.
\end{equation*}

The Christoffel symbols 
\begin{equation*}
  \Gamma_\sigma:=\bigl(\Gamma^a_{\sigma c}\bigr)_{a,c},
  \qquad
  \Gamma_\beta:=\bigl(\Gamma^a_{\beta c}\bigr)_{a,c}
\end{equation*}
have the first nontrivial expansions
\begin{equation}
  \Gamma_\sigma=
  \begin{pmatrix}
    \dfrac{5-6\cos\sigma}{4}r
      &\sqrt2\,r\sin\sigma\cos\sigma
      &r^{-2}+\dfrac{\sin\sigma}{2r}\\[2mm]
    \sqrt2(\cos\sigma-2)r^2
      &r\cos\sigma
      &-\sqrt2\,\sin\sigma\\[2mm]
    -r^2
      &\sqrt2\,r^3\sin^2\sigma
      &\dfrac{2\cos\sigma-5}{4}r
  \end{pmatrix}
  +  \Gamma_{\sigma,R},
  \label{eq:Gamma-sigma-first-expansion}
\end{equation}
and
\begin{equation}
  \Gamma_\beta=
  \begin{pmatrix}
    \sqrt2\,r\sin\sigma\cos\sigma 
      &-\dfrac{\cos\sigma}{r}
        +\dfrac{\sin\sigma(5-6\cos\sigma)}2
      &\dfrac{2\sqrt2\,\cos\sigma(1-\cos\sigma)}{r}\\[2mm]
    r\cos\sigma
      &\sqrt2\,r\sin\sigma(\cos\sigma-2)
      &\dfrac{\sin\sigma}{r}-\dfrac52(1-\cos\sigma)\\[2mm]
    \sqrt2\,r^3\sin^2\sigma
      &-r\sin\sigma
      &2\sqrt2\,r\sin\sigma(1-\cos\sigma)
  \end{pmatrix}
  + \Gamma_{\beta,R}
  \label{eq:Gamma-beta-first-expansion}
\end{equation}
with 
\begin{equation*}
 \Gamma_{\sigma,R}=
  \begin{pmatrix}
    O(r^2)
      &O(r^2)
      &O(1)\\
  O(r^3)
      &O(r^2)
      &O(r)\\
    O(r^3)
      &O(r^4)
      &O(r^2)
  \end{pmatrix}, \quad
   \Gamma_{\beta,R}=
  \begin{pmatrix}
   O(r^2)
      &O(r)
      &O(1)\\
  O(r^2)
      &O(r^2)
      &O(r)\\
  O(r^4)
      &O(r^2)
      &O(r^2)
  \end{pmatrix}.
\end{equation*}
\end{proposition}

\begin{proof}
The formulas from Baldi's construction recorded in Appendix~\ref{app:Baldi-formulas} will be used here.  We first combine the level-set expansion with the angular straightening map to obtain the expansions of $\rho$,
$\zeta$, and $\gamma$.  The first nonradial polynomial and the first
implicit-function coefficient are
\begin{equation}
  P_3(\vartheta)=2\sin\vartheta-\sin(3\vartheta),
  \qquad
  W_1(\vartheta)=-\frac18P_3(\vartheta).
  \label{eq:Baldi-P3-W1}
\end{equation}
Let $\mathfrak c:=4\Kcal(I)$ be Baldi's unscaled level parameter 
and take $\mu:=\sqrt{\mathfrak c}$.  Since $4\Kcal(I)=4I+O(I^3)$ and
$r=\sqrt{2I}$,
\begin{equation}
  \mu=\sqrt2\,r+O(r^5).
  \label{eq:Baldi-mu-r-relation}
\end{equation}
Baldi's implicit level-set expansion has
\begin{equation*}
  \upsilon:=\gamma^{\rm mer}_{\mathfrak c}(\vartheta),
  \qquad
  w(\vartheta,\mu)
  =\frac1{\sqrt2}+W_1(\vartheta)\mu+O(\mu^2),
  \qquad
  \sqrt{2\upsilon}=\mu w(\vartheta,\mu).
\end{equation*}
Applying \eqref{eq:Baldi-P3-W1} and
\eqref{eq:Baldi-mu-r-relation}, we infer
\begin{equation}
  \sqrt{2\upsilon}
  =\frac{\mu}{\sqrt2}-\frac18P_3(\vartheta)\mu^2+O(\mu^3)
  =r+r^2\left(\frac{\sin\vartheta}{4}-\sin^3\vartheta\right)+O(r^3).
  \label{eq:Baldi-radial-before-angle}
\end{equation}
We next derive the inverse angular reparametrization.  Baldi's meridional
level-set function $\gamma^{\rm mer}_{\mathfrak c}(\vartheta)$ satisfies
\begin{equation*}
  \partial_{\mathfrak c}\gamma^{\rm mer}_{\mathfrak c}(\vartheta)
  =\frac14-\frac{3\mu}{16\sqrt2}P_3(\vartheta)+O(\mu^2).
\end{equation*}
With
$F_{\mathfrak c}(\vartheta)
  =\int_0^\vartheta
    \partial_{\mathfrak c}
      \gamma^{\rm mer}_{\mathfrak c}(u)\,\dd u$
and
$f_{\mathfrak c}(\vartheta)
  =2\pi F_{\mathfrak c}(\vartheta)/F_{\mathfrak c}(2\pi)$, it follows from the zero average of $P_3$ that
\begin{align*}
  F_{\mathfrak c}(2\pi)&=\frac\pi2+O(\mu^2),
\\
  f_{\mathfrak c}(\vartheta)
    &=\vartheta-
      \frac{3\mu}{4\sqrt2}
      \int_0^\vartheta P_3(u)\,\dd u+O(\mu^2)
  \\
    &=\vartheta-
      \frac{\mu}{4\sqrt2}
      \bigl(5-6\cos\vartheta+\cos(3\vartheta)\bigr)
      +O(\mu^2),
\end{align*}
since
\begin{equation*}
  \int_0^\vartheta P_3(u)\,\dd u
  =\frac53-2\cos\vartheta+\frac13\cos(3\vartheta).
\end{equation*}
The inverse map $g_{\mathfrak c}=f_{\mathfrak c}^{-1}$ consequently satisfies
\begin{equation*}
  g_{\mathfrak c}(\sigma)
  =\sigma+
    \frac{\mu}{4\sqrt2}
    \bigl(5-6\cos\sigma+\cos(3\sigma)\bigr)
    +O(\mu^2).
\end{equation*}
Evaluating \eqref{eq:Baldi-radial-before-angle} at
$\vartheta=g_{\mathfrak c}(\sigma)$ and applying
\eqref{eq:Baldi-mu-r-relation} we get
  \begin{align*}
  \sqrt{2\upsilon}
    &=r+r^2\left(\frac{\sin\sigma}{4}-\sin^3\sigma\right)+O(r^3),
\\
  \vartheta
    &=\sigma+\frac r4\bigl(5-6\cos\sigma+\cos(3\sigma)\bigr)+O(r^2).
\end{align*}
Substitution into Baldi's formulas
\begin{equation}
  \rho=1+\sqrt{2\upsilon}\sin\vartheta,
  \qquad
  \zeta=\frac{\sqrt{2\upsilon}\cos\vartheta}{\rho}
  \label{eq:Baldi-rho-zeta-exact}
\end{equation}
implies \eqref{eq:rho-first-expansion} and
\eqref{eq:zeta-first-expansion}.  For the azimuthal correction, 
Baldi's formula together with 
\eqref{eq:Baldi-rho-zeta-exact} yields
\begin{equation*}
  Q(\sigma,I)
  =\frac{\sqrt{\mathcal H(\mathfrak c)}}{\rho(\sigma,I)^2}.
\end{equation*}
His expansion of $\mathcal H(\mathfrak c)$ together with
\eqref{eq:rho-first-expansion} gives
\begin{equation*}
  \sqrt{\mathcal H(\mathfrak c)}
    =2\sqrt2\,r+O(r^3),
  \qquad
  \rho(\sigma,I)^{-2}=1-2r\sin\sigma+O(r^2).
\end{equation*}
It then follows
  \begin{equation*}
  Q(\sigma,I)
    =2\sqrt2\,r-4\sqrt2\,r^2\sin\sigma+O(r^3),
  \qquad
  Q_0(I)=2\sqrt2\,r+O(r^3).
\end{equation*}
The term $\gamma(\sigma,I)$ is defined through
\(\partial_\sigma\gamma(\sigma,I)
 =\frac{Q(\sigma,I)-Q_0(I)}{\mathfrak c_I(I)}\) in Baldi's work.
Normalizing it to vanish at $\sigma=0$, we
therefore have
\begin{equation*}
  \begin{aligned}
  \gamma(\sigma,I)
    &=\frac{1}{\mathfrak c_I}
      \int_0^\sigma\bigl(Q(\tilde \sigma,I)-Q_0(I)\bigr)\,\dd \tilde \sigma\\
    &=\frac{1}{4}\int_0^\sigma
      \bigl(Q(\tilde \sigma,I)-Q_0(I)\bigr)\,\dd \tilde \sigma+O(r^3)\\
    &=\sqrt2\,r^2(\cos\sigma-1)+O(r^3)
  \end{aligned}
\end{equation*}
which proves \eqref{eq:ABC-first-expansion}.

For later simplifications, it follows from direct computation that
\begin{equation}\label{eq:AB-identities-1}
\begin{split}
 & (\sin\sigma)A+(\cos\sigma)B=-\frac34\sin\sigma, \quad
  (\cos\sigma)A-(\sin\sigma)B=\frac54(1-\cos\sigma),\\
&  (\cos\sigma)A'-(\sin\sigma)B'=\frac12\sin\sigma, \\
&  2\bigl((\cos\sigma)A-(\sin\sigma)B\bigr) 
    +(\sin\sigma)A'+(\cos\sigma)B'
    =\frac{5-8\cos\sigma}{4}.
    \end{split}
\end{equation}
In particular,
\begin{equation*}
  (\cos\sigma)A'-(\sin\sigma)B'
  +2\bigl((\sin\sigma)A+(\cos\sigma)B\bigr)+\sin\sigma=0,
\end{equation*}
which is a consequence of $\det D\Phi=1$. The equation \eqref{eq:metric-first-expansion} follows by
inserting \eqref{eq:rho-first-expansion}--\eqref{eq:gamma-first-expansion}
into \eqref{eq:metric-components-1}.

Viewing $\mathfrak h:=\dd\rho^2+\dd\zeta^2$ as a two-dimensional metric in $(\sigma,I)$ and denoting
$(\mathfrak h^{ij})=(\mathfrak h_{ij})^{-1}$, 
\eqref{eq:metric-compact} and $\det g=1$ imply
\begin{equation*}
  \det\mathfrak h=\rho^{-2}.
\end{equation*}
For $i,j\in\{\sigma,I\}$, we obtain the identities from block inversion
\begin{equation*}
  g^{ij}=\mathfrak h^{ij},
  \qquad
  g^{i\beta}=-\mathfrak h^{ij}\gamma_j,
  \qquad
  g^{\beta\beta}=\rho^{-2}+\mathfrak h^{ij}\gamma_i\gamma_j.
\end{equation*}
Equations in \eqref{eq:AB-identities-1} then imply \eqref{eq:inverse-metric-first-expansion}.

Next we compute the Christoffel symbols.  Recall
$\varphi=\beta+\gamma(\sigma,I)$ and 
$(\mathbf e_\rho,\mathbf e_\varphi,\mathbf e_z)$ is the standard orthonormal
cylindrical frame.  The coordinate frame is
\begin{equation*}
  e_\sigma
    =\rho_\sigma\mathbf e_\rho
      +\rho\gamma_\sigma\mathbf e_\varphi
      +\zeta_\sigma\mathbf e_z, \quad
  e_\beta=\rho\mathbf e_\varphi, \quad
 e_I=\rho_I\mathbf e_\rho
      +\rho\gamma_I\mathbf e_\varphi
      +\zeta_I\mathbf e_z.
\end{equation*}
The volume identity $\det D\Phi=1$ becomes
\begin{equation}
  J:=\rho_\sigma\zeta_I-\rho_I\zeta_\sigma=\rho^{-1}.
  \label{eq:meridional-Jacobian-exact}
\end{equation}
Consequently, 
$V=V_\rho\mathbf e_\rho+V_\varphi\mathbf e_\varphi+V_z\mathbf e_z=V^ae_a$
has contravariant coordinate components 
\begin{equation} \label{eq:cylindrical-to-coordinate-components-1}
\begin{split}
  V^\sigma
    &=\rho\bigl(\zeta_I V_\rho-\rho_I V_z\bigr), \quad
  V^I=\rho\bigl(-\zeta_\sigma V_\rho+\rho_\sigma V_z\bigr),\\
  V^\beta
    &=\frac{V_\varphi}{\rho}
      -\gamma_\sigma V^\sigma-\gamma_I V^I.
 \end{split}
\end{equation}

Using $\partial_b\mathbf e_\rho=(\partial_b\varphi)\mathbf e_\varphi$, 
$\partial_b\mathbf e_\varphi=-(\partial_b\varphi)\mathbf e_\rho$ and $U^I=0$, the frame
derivatives  are
\begin{equation}  \label{eq:frame-derivative-sigma-sigma}
 \begin{split}
  \partial_\sigma e_\sigma
    &=(\rho_{\sigma\sigma}-\rho\gamma_\sigma^2)\mathbf e_\rho
      +(2\rho_\sigma\gamma_\sigma
        +\rho\gamma_{\sigma\sigma})\mathbf e_\varphi
      +\zeta_{\sigma\sigma}\mathbf e_z,\\
  \partial_\sigma e_\beta
    &=-\rho\gamma_\sigma\mathbf e_\rho
      +\rho_\sigma\mathbf e_\varphi,\\
  \partial_\sigma e_I
    &=(\rho_{\sigma I}-\rho\gamma_\sigma\gamma_I)\mathbf e_\rho
      +(\rho_I\gamma_\sigma+\rho_\sigma\gamma_I
        +\rho\gamma_{\sigma I})\mathbf e_\varphi
      +\zeta_{\sigma I}\mathbf e_z,\\
  \partial_\beta e_\sigma
    &=-\rho\gamma_\sigma\mathbf e_\rho
      +\rho_\sigma\mathbf e_\varphi,\\
  \partial_\beta e_\beta&=-\rho\mathbf e_\rho,\\
  \partial_\beta e_I
    &=-\rho\gamma_I\mathbf e_\rho+\rho_I\mathbf e_\varphi.
 \end{split}
\end{equation}
The coordinate expansions needed in these identities are
\begin{align*}
  \rho_\sigma&=r\cos\sigma+r^2A'+O(r^3),
  &\zeta_\sigma&=-r\sin\sigma+r^2B'+O(r^3),
  &\gamma_\sigma&=r^2C'+O(r^3),
\\
  \rho_I&=\frac{\sin\sigma}{r}+2A+O(r),
  &\zeta_I&=\frac{\cos\sigma}{r}+2B+O(r),
  &\gamma_I&=2C+O(r).
\end{align*}
Since
$\partial_b e_c=\Gamma^a_{bc}e_a$ in the Euclidean ambient space,
applying \eqref{eq:cylindrical-to-coordinate-components-1} to
\eqref{eq:frame-derivative-sigma-sigma} and then using
\eqref{eq:AB-identities-1} shows \eqref{eq:Gamma-sigma-first-expansion} and
\eqref{eq:Gamma-beta-first-expansion}.  For example, this calculation yields
  \begin{equation*}
  \Gamma^\sigma_{\sigma I}
    =r^{-2}+\frac{\sin\sigma}{2r}+O(1),
  \quad
  \Gamma^\sigma_{\beta\beta}
    =-\frac{\cos\sigma}{r}
      +\frac{\sin\sigma(5-6\cos\sigma)}2+O(r),
  \quad
  \Gamma^\beta_{\sigma I}=-\sqrt2\,\sin\sigma+O(r).
\end{equation*}
We note that the trace
\begin{equation*}
  \Gamma^\sigma_{\sigma\sigma}
  +\Gamma^\beta_{\sigma\beta}
  +\Gamma^I_{\sigma I}=0,
  \qquad
  \Gamma^\sigma_{\beta\sigma}
  +\Gamma^\beta_{\beta\beta}
  +\Gamma^I_{\beta I}=0
\end{equation*}
hold at the displayed orders, as required by $\Gamma^a_{ab}=\partial_b\log\sqrt{\det g}=0$.
\end{proof}

\subsection{Interior active action--angle band}

The action--angle chart degenerates at the central circle $I=0$, while the
localized equilibrium becomes inactive where the cutoff vanishes.  The
following corollary isolates a compact toroidal band away from both effects,
providing the uniformly analytic, volume-preserving coordinate domain used in
the linearized and geometric-optics arguments below.

\begin{corollary}
\label{cor:interior-action-angle-band}
Let $U$ be the localized Gavrilov equilibrium and let $\Phi$ be Baldi's
action--angle map.  Define the open active action set
\begin{equation*}
  G_\chi
  :=\bigl\{I\in(0,I^*):\chi(\Kcal(I))\ne0\bigr\}.
\end{equation*}
For every closed interval
\begin{equation}
  [I_-,I_+]\Subset G_\chi,
  \qquad 0<I_-<I_+<I^*,
  \label{eq:interior-action-interval}
\end{equation}
the set
\begin{equation*}
  \mathcal A:=\Phi\bigl(\T^2\times[I_-,I_+]\bigr)
\end{equation*}
is a closed toroidal band, and the restriction of $\Phi$ is an analytic
volume-preserving diffeomorphism of manifolds with boundary
\begin{equation}
  \Phi:\T^2\times[I_-,I_+]\longrightarrow \mathcal A\subset\R^3
  \label{eq:Phi}
\end{equation}
such that $\det D\Phi=1$
and the pullback of the equilibrium velocity has contravariant components
\begin{equation}
  U^a(y)=\bigl(\Omega_1(I),\Omega_2(I),0\bigr)^a,
  \qquad y=(\sigma,\beta,I).
  \label{eq:U-action-angle}
\end{equation}
In particular, the cutoff appearing in the Gavrilov construction is nonzero
throughout the selected band.
\end{corollary}

\begin{proof}
Baldi's Theorem~1.1 provides an analytic diffeomorphism
$\Phi:\T^2\times(0,I^*)\to S^*$ satisfying
$\det D\Phi=1$ and associates the particle dynamics to
\eqref{eq:Baldi-integrable-system}.  It also shows
$P_0\circ\Phi=\Kcal(I)$ with $\Kcal$ strictly increasing, and
\begin{equation*}
  \Omega_1(I)=\chi(\Kcal(I))\Kcal'(I),
  \qquad
  \Omega_2(I)=\sqrt I\,\mathcal R(I)\Omega_1(I).
\end{equation*}
Since the chosen localized equilibrium is nontrivial, $G_\chi$ is a
nonempty open subset of $(0,I^*)$. We obtain \eqref{eq:Phi}--\eqref{eq:U-action-angle} by restricting Baldi's map to any interval
satisfying \eqref{eq:interior-action-interval}.  Compact containment in
$(0,I^*)$ also shows that this restriction extends analytically to an open
neighborhood of the closed band, which verifies the asserted analytic
diffeomorphism in the category of manifolds with boundary.  Finally,
$[I_-,I_+]\subset G_\chi$ implies the nonvanishing of the cutoff.
\end{proof}

\subsection{The linearized velocity operator}
\label{sec-linear}

For $\varepsilon>0$, consider the perturbation 
\begin{equation*}
  u=U+\varepsilon v,
  \qquad
  p=P+\varepsilon q+O(\varepsilon^2).
\end{equation*}
The linearization of \eqref{eq:euler} about $(U,P)$ is
\begin{equation}
  \partial_t v+\nabla_Uv+\nabla_vU+\nabla q=0,
  \qquad \Div v=0
  \label{eq:linearized-velocity-invariant}
\end{equation}
where $\nabla$ denotes the Euclidean Levi--Civita connection.

For a vector field $w=w^ae_a$, the covariant derivative is
\begin{equation*}
  (\nabla_vw)^a=v^b\partial_bw^a+
  \Gamma^a_{bc}v^bw^c.
\end{equation*}
Applying \eqref{eq:U-action-angle}, the transport and deformation terms in
\eqref{eq:linearized-velocity-invariant} are therefore
\begin{equation}
  (\nabla_Uv+\nabla_vU)^a
  =\Om\cdot\nabla_\theta v^a
   +v^I\partial_IU^a
   +2\Gamma^a_{bc}U^bv^c
  \label{eq:transport-deformation-coordinate}
\end{equation}
with
\begin{equation*}
\Om=(\Omega_1, \Omega_2), \quad
  \partial_IU^a
  =\bigl(\Omega_1'(I),\Omega_2'(I),0\bigr)^a.
\end{equation*}
As a consequence of Proposition~\ref{prop:first-geometric-expansion}, the connection contribution to the linearized operator is
\begin{equation*}
  2\Gamma^a_{bc}U^b v^c
  =2\left[
    \Omega_1(I)\Gamma_\sigma+
    \Omega_2(I)\Gamma_\beta
  \right]^a{}_c v^c.
\end{equation*}
Thus the coordinate form of the linearized equation is
\begin{equation*}
  \partial_t v^a
  +\Om\cdot\nabla_\theta v^a
  +v^I\partial_IU^a
  +2\Gamma^a_{bc}U^bv^c
  +g^{ab}\partial_bq=0,
  \qquad
  \partial_av^a=0.
\end{equation*}

Let $\PP_g$ denote the pullback of the Leray projection to the
action--angle chart.  Formally,
\begin{equation}
  \PP_gF=F-\grad_g\Delta_g^{-1}\Div_gF.
  \label{eq:Leray}
\end{equation}
The linearized velocity operator is
\begin{equation*}
  \partial_tv=L_vv,
  \qquad
  L_vv=-\PP_g\bigl(\nabla_Uv+\nabla_vU\bigr).
\end{equation*}

For an axisymmetric perturbation, due to independency of $\beta$,  equation \eqref{eq:transport-deformation-coordinate} becomes
\begin{equation}
  (\nabla_Uv+\nabla_vU)^a
  =\Omega_1(I)\partial_\sigma v^a
   +v^I\partial_IU^a
   +2\Gamma^a_{bc}U^bv^c,
  \label{eq:axisymmetric-transport-deformation}
\end{equation}
with constraint
\begin{equation*}
  \partial_\sigma v^\sigma+\partial_Iv^I=0.
\end{equation*}

\section{Angular modes and the localized frequency profile}
\label{sec-cutoff}

For an angular Fourier mode
\begin{equation*}
  e^{ik\cdot\theta}
  =e^{i(n\sigma+m\beta)},
  \qquad k=(n,m)\in\mathbb Z^2,
\end{equation*}
the principal transport operator becomes
\begin{equation*}
  \partial_t+i\lambda_k(I),
  \qquad
  \lambda_k(I)=k\cdot\Om(I)
  =n\Omega_1(I)+m\Omega_2(I).
\end{equation*}
The mode is nondegenerately sheared on a band if
\begin{equation}
  |\lambda_k'(I)|
  =|n\Omega_1'(I)+m\Omega_2'(I)|\ge c_0>0.
  \label{eq:twist-condition}
\end{equation}
For an axisymmetric mode $m=0$, condition \eqref{eq:twist-condition} reduces to
\begin{equation*}
  |\Omega_1'(I)|\ge c_0>0.
\end{equation*}

Recall Baldi's notation,
\begin{equation*}
  \Omega_1(I)=\chi(\Kcal(I))\Kcal'(I),
  \qquad
  \Kcal(I)=I+\frac{1065}{1024}I^3+O(I^4),
\end{equation*}
hence
\begin{equation*}
  \Omega_1'(I)
  =\chi'(\Kcal(I))(\Kcal'(I))^2
   +\chi(\Kcal(I))\Kcal''(I),
\end{equation*}
where
\begin{equation*}
  \Kcal''(I)=\frac{6390}{1024}I+O(I^2)>0
\end{equation*}
for sufficiently small $I>0$.  If $\chi$ is positive and constant on a
plateau containing $[I_-,I_+]$ in action space, then
\begin{equation*}
  \Omega_1'(I)=\chi\,\Kcal''(I)>0
\end{equation*}
on a sufficiently small interior band.  Hence a standard cutoff with a
positive plateau gives nondegenerate shearing even for axisymmetric modes
$n\ne0$.

We now make the cutoff and the interior band concrete.  Since $\Kcal$ is
analytic,
$\Kcal(I)=I+(1065/1024)I^3+O(I^4)$, there are constants
$I_{\mathrm{an}}>0$ and $C_{\Kcal}>0$ such that
\begin{equation}
  \left|\Kcal''(I)-\frac{3195}{512}I\right|
  \le C_{\Kcal}I^2,
  \qquad 0\le I\le I_{\mathrm{an}}.
  \label{eq:K-second-derivative-remainder}
\end{equation}
Let $p_{\mathrm{adm}}>0$ denote the universal upper endpoint for admissible
cutoff supports provided by Remark~1.2 of \cite{Baldi2024}.  Choose $\iota>0$
so small that
\begin{equation}
  \begin{gathered}
    \frac52\iota<\min\{I^*,I_{\mathrm{an}}\},
    \qquad
    5C_{\Kcal}\iota\le\frac{3195}{512},\\
    \Kcal(5\iota/2)<p_{\mathrm{adm}}.
  \end{gathered}
  \label{eq:iota-smallness}
\end{equation}

\begin{lemma}[Uniform small-action convexity]
\label{lem:uniform-small-action-convexity}
For the choice of $\iota$ in \eqref{eq:iota-smallness}, we have
\begin{equation}
  \Kcal''(I)\ge\frac{3195}{1024}I>0,
  \qquad 0<I\le\frac52\iota.
  \label{eq:K-second-derivative-lower-bound}
\end{equation}
\end{lemma}

\begin{proof}
It follows from equations \eqref{eq:K-second-derivative-remainder} and
\eqref{eq:iota-smallness} that for $0<I\le5\iota/2$,
\[
  \Kcal''(I)
  \ge I\left(\frac{3195}{512}-C_{\Kcal}I\right)
  \ge I\left(\frac{3195}{512}
    -\frac52C_{\Kcal}\iota\right)
  \ge\frac{3195}{1024}I.
\]
\end{proof}

Since $\Kcal$ is strictly increasing, the uncut-pressure values
\begin{equation*}
  p_0=\Kcal(\iota/2),\qquad
  p_1=\Kcal(3\iota/4),\qquad
  p_2=\Kcal(9\iota/4),\qquad
  p_3=\Kcal(5\iota/2)
\end{equation*}
satisfy $0<p_0<p_1<p_2<p_3$.

For an explicit smooth step, define
\begin{equation}
  E(s):=
  \begin{cases}
    0,&s\le0,\\
    e^{-1/s},&s>0,
  \end{cases}
  \qquad
  H(s):=\frac{E(s)}{E(s)+E(1-s)}.
  \label{eq:smooth-step}
\end{equation}
Thus $H=0$ on $(-\infty,0]$ and $H=1$ on $[1,\infty)$.  Fix a
plateau height $\chi_*>0$ and choose the Gavrilov cutoff to be
\begin{equation}
  \chi_\iota(p)
  :=\chi_*
  H\!\left(\frac{p-p_0}{p_1-p_0}\right)
  H\!\left(\frac{p_3-p}{p_3-p_2}\right).
  \label{eq:explicit-pressure-cutoff}
\end{equation}
This is an admissible $C^\infty$ cutoff in $P_0$: its support is contained in
$[p_0,p_3]\Subset(0,p_{\mathrm{adm}})$, and
\begin{equation*}
  \chi_\iota(p)=\chi_*
  \quad\hbox{for }p\in[p_1,p_2].
\end{equation*}
Thus the cutoff has a positive plateau in action space on $[3\iota/4,9\iota/4]$.

\section{Floquet growth for the bicharacteristic--amplitude system}
\label{sec-Floquet-growth}
In this section, we show that the bicharacteristic--amplitude system on a cutoff-generated critical torus has a Floquet multiplier outside the unit circle and hence an exponentially growing velocity amplitude.  We first express
the ray and amplitude equations in action--angle variables and compute the
physical velocity gradient.  We then reduce the constrained amplitude equation
to an planar cocycle, identify and solve its neutral normal limit, and
compute the first normal correction.  Finally, we evaluate the averaged
correction on the critical torus and obtain a hyperbolic Floquet multiplier.

\subsection{Action--angle rays and the bicharacteristic--amplitude system}

Consider a rapidly oscillating perturbation of the form
\begin{equation*}
  v^\varepsilon(y,t)
  \sim a(y,t)e^{iS(y,t)/\varepsilon}.
\end{equation*}
The eikonal equation associated with the principal transport term is
\begin{equation}
  \partial_tS+\Om(I)\cdot\nabla_\theta S=0.
  \label{eq:eikonal}
\end{equation}
For initial phase
\begin{equation*}
  S(0,\theta,I)=k\cdot\theta+\ell I,
\end{equation*}
the solution is explicit:
\begin{equation}
  S(t,\theta,I)
  =k\cdot\theta+\ell I-t\,k\cdot\Om(I).
  \label{eq:phase-solution}
\end{equation}
Let $q=\dd_y S$ denote the coordinate wave covector.  It follows from \eqref{eq:phase-solution} that
\begin{equation}
  q_\theta(t)=k,
  \qquad
  q_I(t)=\ell-t\,k\cdot\Om'(I).
  \label{eq:wave-covector}
\end{equation}

The corresponding velocity-amplitude equation is the classical
bicharacteristic-amplitude system from \cite{FriedlanderVishik1991, LifschitzHameiri1991}, given by
\begin{align}
  \dot x&=U(x),
  \label{eq:LH-x}\\
  \dot\xi&=-\bigl(DU(x)\bigr)^{\mathsf T}\xi,
  \label{eq:LH-xi}\\
  \dot a&=-DU(x)a
  +2\frac{\xi\cdot(DU(x)a)}{|\xi|^2}\,\xi,
  \qquad a\cdot\xi=0,
  \label{eq:LH-a}
\end{align}
where $x(t)$ denotes the trajectory and $\xi=\dd_x S=\mathsf E(y)^{-\mathsf T}q$ is the physical covector.

We now compute the velocity gradient in the bicharacteristic-amplitude system.
Fix
\begin{equation}
  y_0=(\theta_0,I_0),\qquad
  y(t)=(\theta_0+t\Om(I_0),I_0),\qquad
  x(t)=\Phi(y(t)),
  \label{eq:selected-torus-trajectory}
\end{equation}
and denote
\begin{equation}
  \begin{aligned}
    \mathsf E(y)&=D_y\Phi(y),\\
    \mathsf S(I)&:=
    \begin{pmatrix}
      0&0&\Omega_1'(I)\\
      0&0&\Omega_2'(I)\\
      0&0&0
    \end{pmatrix},\\
    \mathsf C(y)&:=\Omega_1(I)\Gamma_\sigma(y)
      +\Omega_2(I)\Gamma_\beta(y).
  \end{aligned}
  \label{eq:E-S-C-definitions}
\end{equation}
For a physical vector $V=V^a(\theta, I)e_a$, 
we have
\begin{equation*}
  V=\mathsf E(y)(V^a)_{a=1}^3,
  \qquad
  \nabla_VU
  =\mathsf E(y)\bigl[(\mathsf S(I)+\mathsf C(y))(V^a)_{a=1}^3\bigr].
\end{equation*}
It follows that along the selected invariant torus,
\begin{equation}
  DU(x(t))
  =\mathsf E(y(t))
   \bigl[\mathsf S(I_0)+\mathsf C(y(t))\bigr]
   \mathsf E(y(t))^{-1}.
  \label{eq:DU-selected-torus}
\end{equation}

In the moving cylindrical frame $(e_\rho,e_\varphi,e_z)$, denote
\begin{equation}
  V_\rho=\Omega_1\rho_\sigma,\qquad
  V_\varphi=\rho(\Omega_2+\Omega_1\gamma_\sigma),\qquad
  V_z=\Omega_1\zeta_\sigma.
  \label{eq:physical-velocity-components}
\end{equation}
It follows from the volume identity \eqref{eq:meridional-Jacobian-exact} that
\begin{equation*}
  D_\rho=\rho(\zeta_I\partial_\sigma-\zeta_\sigma\partial_I),\qquad
  D_z=\rho(-\rho_I\partial_\sigma+\rho_\sigma\partial_I).
\end{equation*}
Since the flow is axisymmetric, its physical velocity gradient is 
\begin{equation}
  [DU]_{\mathrm{cyl}}=
  \begin{pmatrix}
    D_\rho V_\rho&-V_\varphi/\rho&D_zV_\rho\\
    D_\rho V_\varphi&V_\rho/\rho&D_zV_\varphi\\
    D_\rho V_z&0&D_zV_z
  \end{pmatrix}.
  \label{eq:DU-cylindrical-exact}
\end{equation}

More precisely, on the plateau, \eqref{eq:Omega2},
\eqref{eq:Baldi-expansion-summary}, and
Proposition~\ref{prop:first-geometric-expansion} imply
\begin{equation}
  \Omega_1=\chi_*+O(r_0^4),\qquad
  \Omega_2=\frac{\chi_*}{\sqrt2}r_0+O(r_0^3), \quad r_0=\sqrt{2I_0}
  \label{eq:plateau-frequency-r-expansion}
\end{equation}
and we have in the moving cylindrical frame,
\begin{equation}
  [DU(x(t))]_{\mathrm{cyl}}
  =\chi_*
  \begin{pmatrix}
    0&0&1\\
    \frac{\sin\sigma(t)}{\sqrt2}&0&\frac{\cos\sigma(t)}{\sqrt2}\\
    -1&0&0
  \end{pmatrix}
  +O(r_0),\qquad
  \sigma(t)=\sigma_0+\Omega_1(I_0)t.
  \label{eq:DU-cylindrical-leading}
\end{equation}
The remainder is uniform in $t$ and in the initial angles.  Thus, the leading coefficients of $DU$ are
periodic with meridional period $2\pi/|\Omega_1(I_0)|$ in the bounded rotating frame.

\begin{remark}
Let $X_t$ be the physical flow of $U$, set
$F(t):=DX_t(x_0)$.  Then
$\dot F=DU(x(t))F$ and $F(0)=\mathrm{Id}$.  If
$W=\nabla\times U$ and $b=\xi\times a$, a direct calculation from
\eqref{eq:LH-xi}--\eqref{eq:LH-a} gives
\begin{equation}
  \dot b=DU(x(t))b+(W(x(t))\cdot\xi(t))a(t),
  \qquad
  \frac{\dd}{\dd t}(W(x(t))\cdot\xi(t))=0.
  \label{eq:correct-LH-vorticity-equation}
\end{equation}
Thus, when $W\cdot\xi=0$, one has $b(t)=F(t)b(0)$.  However, $b(t)=F(t)b(0)$ is not true for general rays. 
\end{remark}

\subsection{Reduction to a plane}

The constraint $a\cdot\xi=0$ confines the velocity amplitude to the moving
two-dimensional plane $\xi(t)^\perp$.  The purpose of this
subsection is to convert the constrained three-dimensional
bicharacteristic-amplitude equation into a planar cocycle. The resulting trace-free system isolates the genuinely hyperbolic part of the amplitude
dynamics; its monodromy and determinant will provide the instability
mechanism.  This reduction applies to general rays including
those for which the deformation-gradient shortcut in the preceding remark
fails.

Denote
\begin{equation*}
  \nu(t):=\frac{\xi(t)}{|\xi(t)|}.
\end{equation*}
Choose a smooth orthonormal frame
\begin{equation}
  Q(t)=\bigl(q_1(t),q_2(t)\bigr),\qquad
  Q^{\mathsf T}Q=\mathrm{Id}_2,\qquad Q^{\mathsf T}\nu=0,
  \label{eq:amplitude-plane-frame}
\end{equation}
and take $a=Qz$.

\begin{lemma}[Trace-free polarization cocycle]
\label{lem:exact-trace-free-polarization-cocycle}
Let $(x,\xi,a)$ solve the bicharacteristic-amplitude system
\eqref{eq:LH-x}--\eqref{eq:LH-a} on an interval containing $t=0$ and on
which $\xi(t)\ne0$.  For any smooth orthonormal frame $Q$ satisfying
\eqref{eq:amplitude-plane-frame}, the coordinate amplitude $z$ satisfies
the planar system
\begin{equation}
\dot z=B(t)z,\qquad
  B(t)=-Q^{\mathsf T}DU(x(t))Q-Q^{\mathsf T}\dot Q.
  \label{eq:exact-two-dimensional-amplitude-system}
\end{equation}
with the trace
\begin{equation}
  \operatorname{tr}B
  =\nu\cdot DU(x(t))\nu
  =-\frac{\dd}{\dd t}\log|\xi|.
  \label{eq:B-trace-identity}
\end{equation}
Consequently, the normalized coordinate amplitude
\begin{equation}
  Y(t):=\left(\frac{|\xi(t)|}{|\xi(0)|}\right)^{1/2}z(t)
  \label{eq:normalized-coordinate-amplitude}
\end{equation}
satisfies the trace-free system
\begin{equation}
\dot Y=\mathcal B(t)Y,\qquad
  \mathcal B=B-\frac12(\operatorname{tr}B)\mathrm{Id}_2,
  \qquad \operatorname{tr}\mathcal B=0.
  \label{eq:exact-trace-free-amplitude-system}
\end{equation}
\end{lemma}

\begin{proof}
Such a frame exists on every time interval since the rank-two vector
bundle $\nu(t)^\perp$ over an interval is trivial.

Substituting $a=Qz$ into \eqref{eq:LH-a} and applying $Q^{\mathsf T}$ we get
\[
  \dot z+Q^{\mathsf T}\dot Qz
  =-Q^{\mathsf T}DU(x(t))Qz
\]
which verifies \eqref{eq:exact-two-dimensional-amplitude-system}, where we used $Q^{\mathsf T}\nu=0$.

For the component form, denote
\begin{equation*}
  \alpha_{ij}:=q_i\cdot DU(x(t))q_j,
  \qquad \omega:=q_1\cdot\dot q_2.
\end{equation*}
Since $Q^{\mathsf T}\dot Q$ is skew-symmetric, one has
\begin{equation*}
  B=
  \begin{pmatrix}
    -\alpha_{11}&-\alpha_{12}-\omega\\
    -\alpha_{21}+\omega&-\alpha_{22}
  \end{pmatrix}.
\end{equation*}
The frame $(q_1,q_2,\nu)$ is orthonormal and
$\operatorname{tr}DU(x(t))=0$.  Therefore
\[
  \operatorname{tr}B
  =-q_1\cdot DU(x(t))q_1-q_2\cdot DU(x(t))q_2
  =\nu\cdot DU(x(t))\nu.
\]
On the other hand, we have from \eqref{eq:LH-xi} 
\[
  \frac{\dd}{\dd t}\log|\xi|
  =\nu\cdot\frac{\dot\xi}{|\xi|}
  =-\nu\cdot DU(x(t))^{\mathsf T}\nu
  =-\nu\cdot DU(x(t))\nu,
\]
which implies \eqref{eq:B-trace-identity}.

It then follows from \eqref{eq:B-trace-identity} and \eqref{eq:normalized-coordinate-amplitude} that
\[
  \dot Y
  =\left(B+\frac12\frac{\dd}{\dd t}\log|\xi|\,\mathrm{Id}_2\right)Y
  =\left(B-\frac12(\operatorname{tr}B)\mathrm{Id}_2\right)Y,
\] 
and
\begin{equation*}
  \mathcal B=
  \begin{pmatrix}
    \dfrac{\alpha_{22}-\alpha_{11}}2&-\alpha_{12}-\omega\\[2mm]
    -\alpha_{21}+\omega&\dfrac{\alpha_{11}-\alpha_{22}}2
  \end{pmatrix}.
\end{equation*}
The physical amplitude is reconstructed from the normalized polarization by
\begin{equation*}
  a(t)=\left(\frac{|\xi(0)|}{|\xi(t)|}\right)^{1/2}Q(t)Y(t).
\end{equation*}

Let $\mathcal U_Y(t,0)$ denote the fundamental matrix of the normalized system. Liouville’s formula and $\operatorname{tr}\mathcal B=0$ imply
\[
\det\mathcal U_Y(t,0)
=\exp\!\left(\int_0^t\operatorname{tr}\mathcal B(s)\,ds\right)=1.
\]
Hence the normalized amplitude evolution belongs to $SL(2,\mathbb R)$ and preserves oriented area in the polarization coordinates.
\end{proof}

\begin{lemma}\label{lem:LH-amplitude-area-identity}
For the physical bicharacteristic-amplitude evolution on the two-dimensional
polarization plane $\xi(t)^\perp$, the solution operator satisfies
\begin{equation}
  \bigl|\det\mathcal U_a(t)\bigr|
  =\frac{|\xi(0)|}{|\xi(t)|}.
  \label{eq:LH-amplitude-area-Jacobian}
\end{equation}
\end{lemma}

\begin{proof}
Let $\Psi(t,0)$ denote the coordinate evolution operator for
$\dot z=B(t)z$ in the orthonormal frames $Q(0)$ and $Q(t)$ of
$\xi(0)^\perp$ and $\xi(t)^\perp$, respectively.  Since the frames
preserve Euclidean area,
\begin{equation}
  \bigl|\det\mathcal U_a(t)\bigr|=|\det\Psi(t,0)|.
  \label{eq:physical-coordinate-area-identification}
\end{equation}
It follows from Liouville's formula and \eqref{eq:B-trace-identity} that
\begin{equation*}
  \det\Psi(t,0)
  =\exp\left(\int_0^t\operatorname{tr}B(s)\,\dd s\right)
  =\exp\left(\int_0^t
    \nu(s)\cdot DU(x(s))\nu(s)\,\dd s\right).
\end{equation*}
On the other hand, the wave-covector equation implies
\begin{equation*}
  \frac{\dd}{\dd t}\log|\xi(t)|
  =-\nu(t)\cdot DU(x(t))\nu(t).
\end{equation*}
Consequently,
\[
  \det\Psi(t,0)
  =\exp\bigl(\log|\xi(0)|-\log|\xi(t)|\bigr)
  =\frac{|\xi(0)|}{|\xi(t)|}.
\]
Combining this identity with
\eqref{eq:physical-coordinate-area-identification} concludes 
\eqref{eq:LH-amplitude-area-Jacobian}.
\end{proof}

\subsection{The normal limiting cocycle and its identity monodromy}

In this subsection we isolate and solve the leading amplitude dynamics when the wave covector becomes large in the direction normal to an invariant torus.  The resulting periodic system is a zeroth-order reference cocycle: its identity monodromy shows that cumulative growth must come from the first normal-frequency correction.  Conjugation by its explicit fundamental matrix then exposes the correction responsible for exponential Floquet growth when the normal frequency is frozen at a critical torus.

Denote
\begin{equation*}
  U=U_{\mathrm m}+V_\varphi\mathbf e_\varphi,
  \qquad
  U_{\mathrm m}:=V_\rho\mathbf e_\rho+V_z\mathbf e_z.
\end{equation*}
\begin{lemma}\label{lem:asymptotic-normal-limiting-cocycle}
Let $I_0$ lie on an active regular invariant torus on which
\begin{equation*}
  |\nabla I|>0,
  \qquad \Omega_1(I_0)\ne0,
  \qquad |U_{\mathrm m}|>0.
\end{equation*}
Let $q_0=(n,m,\ell)$ and $k=(n,m)$, and assume
\begin{equation}
  \lambda_k'(I_0)=k\cdot\Om'(I_0)\ne0.
  \label{eq:pointwise-twist-selected-torus}
\end{equation}
Then there is a trace-free coefficient $\mathcal B_\infty(t)$,
periodic in the cylindrical identification with meridional period $T_0:=\frac{2\pi}{|\Omega_1(I_0)|}$,
such that the trace-free coefficient $\mathcal B(t)$ given in \eqref{eq:exact-trace-free-amplitude-system} associated with a particularly chosen frame $Q(t)$ satisfies
\begin{equation}
  \mathcal B(t)=\mathcal B_\infty(t)+O(t^{-1})
  \qquad (t\to+\infty)
  \label{eq:amplitude-coefficient-asymptotics}
\end{equation}
where the remainder is uniform in the initial meridional phase on the selected
torus.
\end{lemma}

\begin{proof}
In physical covector notation, \eqref{eq:wave-covector} becomes
\begin{equation}
  \xi(t)=\bigl(\ell-t\lambda_k'(I_0)\bigr)\nabla I(x(t))
  +\eta_k(t),
  \qquad
  \eta_k(t):=\mathsf E(y(t))^{-\mathsf T}(n,m,0)^{\mathsf T}.
  \label{eq:physical-twisted-covector-decomposition}
\end{equation}
To justify the periodicity, recall from \eqref{eq:Phi-explicit}
that the cylindrical representation of $\mathsf E=D\Phi$ is
\begin{equation*}
  [\mathsf E]_{\mathrm{cyl}}
  =\widehat{\mathsf E}(\sigma,I)
  :=
  \begin{pmatrix}
    \rho_\sigma&0&\rho_I\\
    \rho\gamma_\sigma&\rho&\rho\gamma_I\\
    \zeta_\sigma&0&\zeta_I
  \end{pmatrix}.
\end{equation*}
It is independent of $\beta$, and therefore
\begin{equation}
  [\eta_k(t)]_{\mathrm{cyl}}
  =\widehat{\mathsf E}(\sigma(t),I_0)^{-\mathsf T}
   (n,m,0)^{\mathsf T}.
  \label{eq:eta-k-cylindrical-periodic}
\end{equation}
Since $\det\widehat{\mathsf E}=1$ and its entries are smooth and
$2\pi$-periodic in $\sigma$, it follows from
\eqref{eq:eta-k-cylindrical-periodic} that $[\eta_k(t)]_{\mathrm{cyl}}$ and its
$\sigma$-derivatives are bounded and periodic.  In view of 
$\sigma(t)=\sigma_0+\Omega_1(I_0)t$, the cylindrical components and their time derivatives are also $T_0$-periodic.  The cylindrical frame has angular velocity
\begin{equation*}
  \dot\varphi(t)
  =\Omega_2(I_0)
   +\Omega_1(I_0)\gamma_\sigma(\sigma(t),I_0),
\end{equation*}
which is also bounded and $T_0$-periodic.  Hence  $\dot\eta_k$, when expressed in cylindrical components, is bounded and periodic.  Denote
\[
  s(t):=\ell-t\lambda_k'(I_0),
  \qquad \mathbf n:=\frac{\nabla I}{|\nabla I|},
  \qquad \Pi_T:=\mathrm{Id}-\mathbf n\otimes\mathbf n.
\]
Regularity of the torus implies $|\nabla I|>0$.  Since
$\lambda_k'(I_0)\ne0$, we have $|s(t)|\asymp t$, and the sign
$\operatorname{sgn}s(t)$ is constant for all sufficiently large
$t$.  Expanding $\xi/|\xi|$ in terms of $|s|^{-1}$ we obtain
\begin{equation}
  \nu(t)
  =(\operatorname{sgn}s(t))\mathbf n(t)
   +\frac{\Pi_T(t)\eta_k(t)}
          {|s(t)|\,|\nabla I(x(t))|}
   +O(|s(t)|^{-2}).
  \label{eq:unit-covector-asymptotics}
\end{equation}
It follows from \eqref{eq:unit-covector-asymptotics} that
\begin{equation}
  \nu(t)-(\operatorname{sgn}s(t))\mathbf n(t)=O(t^{-1}),
  \qquad
  \dot\nu(t)-(\operatorname{sgn}s(t))\dot{\mathbf n}(t)=O(t^{-1}).
  \label{eq:unit-covector-C1-asymptotics}
\end{equation}

For sufficiently large $t$, choose the polarization frame
$Q(t)=(q_1(t),q_2(t))$ by taking $q_2$ to be the normalized projection of
$\mathbf e_\varphi$ onto $\nu(t)^\perp$ and setting either
$q_1=q_2\times\nu$ or $q_1=-q_2\times\nu$, with the choice of sign so
that $q_1\to U_{\mathrm m}/|U_{\mathrm m}|$.
Since $\mathbf e_\varphi\cdot\mathbf n=0$, we have
\begin{equation}
  q_2(t)
  =\frac{\mathbf e_\varphi-(\mathbf e_\varphi\cdot\nu)\nu}
  {\sqrt{1-(\mathbf e_\varphi\cdot\nu)^2}}.
  \label{eq:projected-polarization-frame}
\end{equation}
Denote
\begin{equation}\label{eq:limiting-tangent-frame}
  \boldsymbol\tau:=\frac{U_{\mathrm m}}{|U_{\mathrm m}|},
  \qquad Q_\infty:=(\boldsymbol\tau,\mathbf e_\varphi).
\end{equation}
One choice of the sign gives
$q_1\to\boldsymbol\tau$.  It follows from
\eqref{eq:unit-covector-C1-asymptotics} and
\eqref{eq:projected-polarization-frame} that
\begin{equation}
  \|Q(t)-Q_\infty(t)\|
  +\|\dot Q(t)-\dot Q_\infty(t)\|
  \le \frac{C}{1+t}
  \label{eq:polarization-frame-C1-convergence}
\end{equation}
in the moving orthonormal cylindrical frame $(\mathbf e_\rho(t),\mathbf e_\varphi(t),\mathbf e_z)$.

In the same frame, $DU(x(t))$ and $Q_\infty(t)$ are
$T_0$-periodic and uniformly bounded.  Define
\begin{equation*}
  B_\infty(t)
  :=-Q_\infty(t)^{\mathsf T}DU(x(t))Q_\infty(t)
    -Q_\infty(t)^{\mathsf T}\dot Q_\infty(t).
\end{equation*}
Applying $B=-Q^{\mathsf T}DU(x(t))Q-Q^{\mathsf T}\dot Q$ and
\eqref{eq:polarization-frame-C1-convergence}, we obtain
\begin{equation}
  \|B(t)-B_\infty(t)\|
  \le C\bigl(
    \|Q-Q_\infty\|+\|\dot Q-\dot Q_\infty\|
  \bigr)
  \le\frac{C}{1+t}.
  \label{eq:unnormalized-limiting-B-bound}
\end{equation}
Finally, the operation
\[
  B\longmapsto B-\frac12(\operatorname{tr}B)\mathrm{Id}_2
\]
is linear and bounded, which combined with \eqref{eq:unnormalized-limiting-B-bound} implies \eqref{eq:amplitude-coefficient-asymptotics}.  It also shows that $\mathcal B_\infty$ is trace-free and $T_0$-periodic.  
\end{proof}

\begin{lemma}\label{lem:explicit-normal-limiting-cocycle}
On an active regular torus on which $|U_{\mathrm m}|>0$ and
$V_\varphi\ne0$, the coefficient
in Lemma~\ref{lem:asymptotic-normal-limiting-cocycle} is
\begin{equation}
  \mathcal B_\infty(t)=
  \begin{pmatrix}
    -\gamma(t)&\eta(t)\\
    0&\gamma(t)
  \end{pmatrix},
  \label{eq:explicit-trace-free-limiting-B}
\end{equation}
where
\begin{equation}
  \gamma
    =\frac{1}{4|U_{\mathrm m}|^2}
       \frac{\dd}{\dd t}|U_{\mathrm m}|^2
    +\frac{\dot V_\varphi}{2V_\varphi}
    =\frac{1}{4|U_{\mathrm m}|^2}
       \frac{\dd}{\dd t}|U_{\mathrm m}|^2
      -\frac{\dot\rho}{2\rho},
  \qquad
  \eta
  =\frac{1}{V_\varphi|U_{\mathrm m}|}
    \frac{\dd}{\dd t}|U_{\mathrm m}|^2.
  \label{eq:gamma-eta-limiting-system}
\end{equation}
\end{lemma}

\begin{proof}
We compute $\mathcal B_\infty$ indirectly. We shall first show $|U|^2$ and $L:=\rho V_\varphi$ are torus invariants and that allows us to construct two independent solutions of the normal limiting physical amplitude equation, from which we obtain a fundamental matrix $Z_\infty$ and hence $B_\infty=\dot Z_\infty Z_\infty^{-1}$.  

We claim $|U|^2$ and $L$ are constant on the torus.  Indeed, it follows from
\eqref{eq:Gavrilov-first-integrals} and \eqref{eq:localized-pressure-primitive} that
\begin{equation*}
  |U|^2=3\chi(P_0)^2P_0.
\end{equation*}
Since $P_0\circ\Phi=\Kcal(I)$ by
\eqref{eq:Baldi-pressure-action}, it follows that $|U|^2$ depends only on
$I$.  Likewise, \eqref{eq:localized-pressure-action} shows that the localized
  steady pressure satisfies $P=P(I)=\mathcal W(\Kcal(I))$.
For the angular momentum, the azimuthal component of the steady
axisymmetric Euler equation gives
\begin{equation*}
  U\cdot\nabla L
  =V_\rho\partial_\rho(\rho V_\varphi)
   +V_z\partial_z(\rho V_\varphi)=0.
\end{equation*}
Moreover, due to axisymmetry $\partial_\beta L=0$, we have in action--angle
coordinates,
\[
  0=U\cdot\nabla L
   =\Omega_1(I)\partial_\sigma L
    +\Omega_2(I)\partial_\beta L
   =\Omega_1(I)\partial_\sigma L.
\]
Since $\Omega_1(I_0)\ne0$ on the active regular torus, it follows
$\partial_\sigma L=0$, and therefore $L=L(I)$.  Thus, on the selected torus,
the values of $|U|^2$ and $L=L(I_0)$ are constant along the orbit, and
\begin{equation}
  V_\varphi=\frac L\rho,
  \qquad
  |U_{\mathrm m}|^2=|U|^2-\frac{L^2}{\rho^2}.
  \label{eq:meridional-speed-invariants}
\end{equation}
For the small active bands used above, we have $L\ne0$ after decreasing $\iota$ if
necessary in view of \eqref{eq:physical-velocity-components} and \eqref{eq:plateau-frequency-r-expansion}; hence the assumption $V_\varphi\ne0$ holds.

Next we construct two solutions of the limiting amplitude equation:
\begin{equation*}
  a_1(t)=U(x(t)),
  \qquad
  a_2(t)
  =\frac{U_{\mathrm m}(x(t))}{|U_{\mathrm m}(x(t))|^2}.
\end{equation*}
We claim they are periodic in the cylindrical coordinates.  Indeed,
$DU(x(t))U=-\nabla P$ is parallel to $\nabla I$.  We have the normal limiting
bicharacteristic-amplitude equation
\begin{equation}
  \dot a=-DU(x(t))a
  +2\mathbf n\bigl(\mathbf n\cdot DU(x(t))a\bigr),
  \qquad a\cdot\mathbf n=0.
  \label{eq:normal-limiting-LH-equation}
\end{equation}
For $a_1=U$, one can verify that $\dot a_1=DU(x(t))U$ and
$\mathbf n(\mathbf n\cdot DU(x(t))U)=DU(x(t))U$; hence $a_1$ is a
solution of \eqref{eq:normal-limiting-LH-equation}.

To verify $a_2$, denote $K:=\rho\mathbf e_\varphi$ and a tangent amplitude $a=xU+yK$.
Axisymmetry shows $[U,K]=0$ and $\dot K=DU(x(t))K$, while
\begin{equation*}
  DU(x(t))K=-V_\varphi\mathbf e_\rho+V_\rho\mathbf e_\varphi.
\end{equation*}
Recall $\Pi_T$ is orthogonal projection onto the tangent plane of the invariant
torus, and hence
\begin{equation}
  \Pi_TDU(x(t))K
  =-\frac{V_\varphi V_\rho}{|U_{\mathrm m}|^2}U
   +\frac{|U|^2V_\rho}{\rho|U_{\mathrm m}|^2}K.
  \label{eq:tangential-DU-K}
\end{equation}
Since $\Pi_TDU(x(t))U=0$, tangential projection of
\eqref{eq:normal-limiting-LH-equation} yields
\begin{equation*}
  \dot x\,U+\dot y\,K=-2y\,\Pi_TDU(x(t))K.
\end{equation*}
In view of \eqref{eq:tangential-DU-K} we infer
\begin{equation}
  \dot x=2\frac{V_\varphi V_\rho}{|U_{\mathrm m}|^2}y,
  \qquad
  \dot y=-2\frac{|U|^2V_\rho}
                   {\rho|U_{\mathrm m}|^2}y.
  \label{eq:limiting-amplitude-UK-system}
\end{equation}
Besides $(x,y)=(1,0)$, this system has the solution
\begin{equation*}
  x=\frac1{|U_{\mathrm m}|^2},
  \qquad
  y=-\frac{L}{\rho^2|U_{\mathrm m}|^2},
\end{equation*}
Indeed, $\dot\rho=V_\rho$ and
\eqref{eq:meridional-speed-invariants} imply
\begin{equation*}
  \frac{\dd}{\dd t}|U_{\mathrm m}|^2
  =\frac{2L^2V_\rho}{\rho^3}
  =\frac{2V_\varphi^2V_\rho}{\rho}.
\end{equation*}
Consequently,
\begin{equation*}
  \dot x
  =-\frac{\frac{\dd}{\dd t}|U_{\mathrm m}|^2}
          {|U_{\mathrm m}|^4}
  =-\frac{2V_\varphi^2V_\rho}
          {\rho|U_{\mathrm m}|^4}
  =2\frac{V_\varphi V_\rho}{|U_{\mathrm m}|^2}y,
\end{equation*}
while
\begin{equation*}
  \frac{\dot y}{y}
  =-2\frac{V_\rho}{\rho}
   -\frac{\frac{\dd}{\dd t}|U_{\mathrm m}|^2}
          {|U_{\mathrm m}|^2}
  =-2\frac{|U|^2V_\rho}{\rho|U_{\mathrm m}|^2}.
\end{equation*}
Thus $(x,y)$ solves
\eqref{eq:limiting-amplitude-UK-system}, and
\begin{equation*}
  xU+yK
  =\frac{U}{|U_{\mathrm m}|^2}
   -\frac{L}{\rho|U_{\mathrm m}|^2}\mathbf e_\varphi
  =\frac{U_{\mathrm m}}{|U_{\mathrm m}|^2}.
\end{equation*}

In the frame \eqref{eq:limiting-tangent-frame}, these two solutions form the
fundamental matrix
\begin{equation}
  Z_\infty(t)=
  \begin{pmatrix}
    |U_{\mathrm m}|&1/|U_{\mathrm m}|\\
    V_\varphi&0
  \end{pmatrix}.
  \label{eq:limiting-fundamental-Z}
\end{equation}
It is invertible since $|U_{\mathrm m}|>0$ and $V_\varphi\ne0$.  Therefore
\begin{equation}
  B_\infty=\dot Z_\infty Z_\infty^{-1}
  =\begin{pmatrix}
    -\dfrac{1}{2|U_{\mathrm m}|^2}
       \dfrac{\dd}{\dd t}|U_{\mathrm m}|^2
    &
    \dfrac{1}{V_\varphi|U_{\mathrm m}|}
       \dfrac{\dd}{\dd t}|U_{\mathrm m}|^2\\[4mm]
    0&\dfrac{\dot V_\varphi}{V_\varphi}
  \end{pmatrix}.
  \label{eq:explicit-limiting-B}
\end{equation}
Since $\det\mathsf E=1$, we have
$e_\sigma\times e_\beta=\nabla I$.  Let
$(e_\sigma)_{\mathrm m}$ denote the meridional part of $e_\sigma$.  It follows from $e_\beta=\rho\mathbf e_\varphi$ that
\begin{equation*}
  |\nabla I|=|e_\sigma\times e_\beta|
  =\rho |(e_\sigma)_{\mathrm m}|.
\end{equation*}
Moreover, $U=\Omega_1e_\sigma+\Omega_2e_\beta$ implies
\begin{equation*}
  U_{\mathrm m}=\Omega_1(e_\sigma)_{\mathrm m},
  \qquad
  |U_{\mathrm m}|
  =|\Omega_1(I_0)|\,|(e_\sigma)_{\mathrm m}|.
\end{equation*}
Combining the equations above yields
\begin{equation*}
  \rho|U_{\mathrm m}|
  =|\Omega_1(I_0)|\,|\nabla I|.
\end{equation*}
Together with $\dot V_\varphi/V_\varphi=-\dot\rho/\rho$, this verifies
\begin{equation*}
  \operatorname{tr}B_\infty
  =-\frac{1}{2|U_{\mathrm m}|^2}
     \frac{\dd}{\dd t}|U_{\mathrm m}|^2
   -\frac{\dot\rho}{\rho}
  =-\frac{\dd}{\dd t}\log|\nabla I(x(t))|,
\end{equation*}
in agreement with \eqref{eq:B-trace-identity}.  Taking the trace-free part
of \eqref{eq:explicit-limiting-B} we obtain
\eqref{eq:explicit-trace-free-limiting-B}--
\eqref{eq:gamma-eta-limiting-system}.  Finally, $|U_{\mathrm m}|^2$ and
$V_\varphi$ are $T_0$-periodic in the cylindrical coordinate.
\end{proof}

\begin{proposition}\label{prop:identity-limiting-monodromy}
The trace-free periodic system \eqref{eq:explicit-trace-free-limiting-B}
has a fundamental matrix $\Phi_\infty(t)$ normalized by
$\Phi_\infty(0)=\mathrm{Id}_2$.  Its monodromy over one meridional period is
the identity:
\begin{equation}
  \Phi_\infty(T_0)=\mathrm{Id}_2.
  \label{eq:identity-limiting-monodromy}
\end{equation}
\end{proposition}

\begin{proof}
Denote
\begin{equation*}
  G_\infty(t):=\int_0^t\gamma(s)\,\dd s
  =\frac14\log
    \frac{|U_{\mathrm m}(x(t))|^2}{|U_{\mathrm m}(x(0))|^2}
   +\frac12\log\frac{|V_\varphi(t)|}{|V_\varphi(0)|}.
\end{equation*}
The normalized fundamental matrix is
\begin{equation*}
  \Phi_\infty(t)=
  \begin{pmatrix}
    e^{-G_\infty(t)}&e^{-G_\infty(t)}J_\infty(t)\\
    0&e^{G_\infty(t)}
  \end{pmatrix},
  \qquad
  J_\infty(t):=\int_0^t\eta(s)e^{2G_\infty(s)}\,\dd s.
\end{equation*}
Since
\begin{equation*}
  e^{2G_\infty(t)}
  =\frac{|U_{\mathrm m}(x(t))|}{|U_{\mathrm m}(x(0))|}
   \frac{V_\varphi(t)}{V_\varphi(0)},
\end{equation*}
we obtain
\begin{equation*}
  J_\infty(t)
  =\frac{|U_{\mathrm m}(x(t))|^2-|U_{\mathrm m}(x(0))|^2}
         {V_\varphi(0)|U_{\mathrm m}(x(0))|}.
\end{equation*}
Since $|U_{\mathrm m}|^2$ and $V_\varphi$ are $T_0$-periodic in the cylindrical coordinate,  we have $G_\infty(T_0)=J_\infty(T_0)=0$, and 
\eqref{eq:identity-limiting-monodromy} follows immediately.
\end{proof}

\subsection{The first normal correction and its averaged generator}
\label{sec-normal}

We now compute the first correction from \eqref{eq:amplitude-coefficient-asymptotics}.  The analysis below for large $|s|$
applies also when $\lambda_k'(I_0)=0$.  

For a fixed ray with $\xi=s|\nabla I|\mathbf n+\eta_k$,  we introduce the periodic tangent
vectors
\begin{equation*}
  \mathbf r_k:=\frac{\Pi_T\eta_k}{|\nabla I|},
  \qquad
  \mathbf c:=\Pi_TDU(x(t))^{\mathsf T}\mathbf n,
  \qquad
  \mathbf d:=\Pi_TDU(x(t))\mathbf n,
\end{equation*}
and their components in the limiting frame
$Q_\infty=(\boldsymbol\tau,\mathbf e_\varphi)$,
\begin{equation}
  \mathsf r:=Q_\infty^{\mathsf T}\mathbf r_k,
  \qquad
  \mathsf c:=Q_\infty^{\mathsf T}\mathbf c,
  \qquad
  \mathsf d:=Q_\infty^{\mathsf T}\mathbf d.
  \label{eq:r-c-d-components}
\end{equation}

To identify the moving plane $\xi^\perp$ with the tangent plane
$\mathbf n^\perp$, we express a physical amplitude in the graph form
\begin{equation}
  a=p-
  \frac{p\cdot\eta_k}
       {s|\nabla I|+\mathbf n\cdot\eta_k}\,\mathbf n,
  \qquad p\cdot\mathbf n=0.
  \label{eq:amplitude-graph-map}
\end{equation}
We express the tangent part $p=Q_\infty z_{\mathrm{tan}}$ in the limiting tangent frame
$Q_\infty=(\boldsymbol\tau,\mathbf e_\varphi)$.

\begin{lemma}\label{lem:first-corrected-tangent-amplitude}
On a fixed compact active band, assume that $|s(t)|$ is sufficiently large.
Let $a$ solve the bicharacteristic-amplitude equation \eqref{eq:LH-a}. The tangent coordinate $z_{\mathrm{tan}}$ satisfies
\begin{equation}
  \dot z_{\mathrm{tan}}
  =\left(B_\infty(t)+\frac1{s(t)}D(t)+R_2(t)\right)z_{\mathrm{tan}}
  \label{eq:first-corrected-tangent-amplitude}
\end{equation}
with
\begin{equation}
  D=(\mathsf d-\mathsf c)\mathsf r^{\mathsf T}
      +2\mathsf r\mathsf c^{\mathsf T}.
  \label{eq:D-first-correction}
\end{equation}
Moreover,
\begin{equation}
  \|R_2(t)\|+\|\partial_tR_2(t)\|\le C|s(t)|^{-2},
  \label{eq:first-corrected-tangent-remainder}
\end{equation}
with a constant $C>0$ uniform on the band.
\end{lemma}

\begin{proof}
Note $z_{\mathrm{tan}}=Q_\infty^{\mathsf T}p=Q_\infty^{\mathsf T}Qz$. 
Denote
\[
  \alpha:=\frac{p\cdot\eta_k}
                {s|\nabla I|+\mathbf n\cdot\eta_k},
  \qquad a=p-\alpha\mathbf n.
\]
Since $p\cdot\mathbf n=0$ and
$\mathbf n\cdot\xi=s|\nabla I|+\mathbf n\cdot\eta_k$, one has
\[
  a\cdot\xi=p\cdot\eta_k
  -\alpha(s|\nabla I|+\mathbf n\cdot\eta_k)=0.
\]
Thus \eqref{eq:amplitude-graph-map} parametrizes $\xi^\perp$, and
$p=\Pi_Ta$.  It follows from
$p\cdot\eta_k=p\cdot\Pi_T\eta_k
=|\nabla I|\,\mathbf r_k\cdot p$ that
\begin{equation}
  \alpha
  =\frac{\mathbf r_k\cdot p}
  {s+(\mathbf n\cdot\eta_k)/|\nabla I|}
  =\frac1s(\mathbf r_k\cdot p)+O(s^{-2})|p|.
  \label{eq:graph-coefficient-first-expansion}
\end{equation}

Differentiate $a=p-\alpha\mathbf n$ and project onto
$\mathbf n^\perp$, and apply $\dot{\mathbf n}\perp\mathbf n$, we deduce
\begin{equation*}
  \Pi_T\dot a=\Pi_T\dot p-\alpha\dot{\mathbf n}.
\end{equation*}
The tangential projection of the right-hand side of \eqref{eq:LH-a} is
\begin{equation*}
  -\Pi_TDU(x(t))p+\alpha\mathbf d
  +2\frac{\xi\cdot DU(x(t))a}{|\xi|^2}
     |\nabla I|\,\mathbf r_k,
\end{equation*}
where we used $\Pi_TDU(x(t))\mathbf n=\mathbf d$ and
$\Pi_T\xi=\Pi_T\eta_k=|\nabla I|\,\mathbf r_k$.  Furthermore, we have
\begin{equation*}
\begin{split}
  \xi\cdot DU(x(t))a
  &=s|\nabla I|\,\mathbf n\cdot DU(x(t))p+O(1)|p|
  =s|\nabla I|(\mathbf c\cdot p)+O(1)|p|,\\
|\xi|^2&=s^2|\nabla I|^2+O(|s|),
  \end{split}
\end{equation*}
applying
$\mathbf n\cdot DU(x(t))p
=(\Pi_TDU(x(t))^{\mathsf T}\mathbf n)\cdot p
=\mathbf c\cdot p$.  Hence
\begin{equation}
  \frac{\xi\cdot DU(x(t))a}{|\xi|^2}
  |\nabla I|\,\mathbf r_k
  =\frac1s\mathbf r_k(\mathbf c\cdot p)+O(s^{-2})|p|.
  \label{eq:projected-LH-rank-one-expansion}
\end{equation}
Combining \eqref{eq:graph-coefficient-first-expansion}--
\eqref{eq:projected-LH-rank-one-expansion}, and applying
\begin{equation*}
  \dot{\mathbf n}
  =-\Pi_TDU(x(t))^{\mathsf T}\mathbf n=-\mathbf c,
\end{equation*}
we obtain
\begin{equation*}
  \Pi_T\dot p
  =-\Pi_TDU(x(t))p
   +\frac1s\left\{
     (\mathbf d-\mathbf c)(\mathbf r_k\cdot p)
     +2\mathbf r_k(\mathbf c\cdot p)
   \right\}
   +O(s^{-2})|p|.
\end{equation*}

Finally, applying $p=Q_\infty z_{\mathrm{tan}}$ we get the leading coefficient 
\[
  -Q_\infty^{\mathsf T}DU(x(t))Q_\infty
  -Q_\infty^{\mathsf T}\dot Q_\infty=B_\infty,
\]
while in terms of quantities of \eqref{eq:r-c-d-components},
\[
  (\mathbf d-\mathbf c)(\mathbf r_k\cdot p)
     +2\mathbf r_k(\mathbf c\cdot p)= (\mathsf d-\mathsf c)\mathsf r^{\mathsf T}
  +2\mathsf r\mathsf c^{\mathsf T}=D.
\]
Combining with the facts
\[Q_\infty^{\mathsf T}Q_\infty=\mathrm{Id}_2,
\qquad
Q_\infty Q_\infty^{\mathsf T}=\Pi_T\]
concludes \eqref{eq:first-corrected-tangent-amplitude}--
\eqref{eq:D-first-correction}. For large $|s|$, Taylor's formula together with compactness of the active band 
leads to the uniform remainder estimate
\eqref{eq:first-corrected-tangent-remainder}.
\end{proof}

\begin{lemma}\label{lem:post-limiting-Floquet-reduction}
On a fixed compact active band, let $z_{\mathrm{tan}}$ solve
\eqref{eq:first-corrected-tangent-amplitude} on an interval on which
$|s(t)|$ is sufficiently large.  
Under the substitution $z_{\mathrm{tan}}=Z_\infty z_{\mathrm F}$, we have
\begin{equation}
  \dot z_{\mathrm F}
  =\frac1{s(t)}K(t)z_{\mathrm F}
   +\widetilde R_2(t)z_{\mathrm F},
  \qquad
  K(t):=Z_\infty(t)^{-1}D(t)Z_\infty(t),
  \label{eq:post-limiting-Floquet-system}
\end{equation}
where $K$ is $T_0$-periodic in the cylindrical coordinate and
\begin{equation}
  \|\widetilde R_2(t)\|\le C|s(t)|^{-2}.
\end{equation}
\end{lemma}

\begin{proof}
Applying $z_{\mathrm{tan}}=Z_\infty z_{\mathrm F}$ and $\dot Z_\infty=B_\infty Z_\infty$, straightforward computation shows
\[
  B_\infty Z_\infty z_{\mathrm F}+Z_\infty\dot z_{\mathrm F}
  =B_\infty Z_\infty z_{\mathrm F}
   +\frac1sDZ_\infty z_{\mathrm F}
   +R_2Z_\infty z_{\mathrm F},
\]
and it follows
\[
  \dot z_{\mathrm F}
  =\frac1sZ_\infty^{-1}DZ_\infty z_{\mathrm F}
   +Z_\infty^{-1}R_2Z_\infty z_{\mathrm F}
\]
which verifies \eqref{eq:post-limiting-Floquet-system} with
\(
  \widetilde R_2:=Z_\infty^{-1}R_2Z_\infty.
\)
Since $Z_\infty$ and $Z_\infty^{-1}$ are periodic and uniformly
bounded on the active band, we have
\[
  \|\widetilde R_2(t)\|
  \le \|Z_\infty(t)^{-1}\|\,\|R_2(t)\|\,\|Z_\infty(t)\|
  \le C'|s(t)|^{-2}.
\]
Finally, since both $D$ and $Z_\infty$ are $T_0$-periodic in the cylindrical coordinate, so is $K=Z_\infty^{-1}DZ_\infty$.
\end{proof}

Define the average
\begin{equation}
  \overline K:=\frac1{T_0}\int_0^{T_0}K(t)\,\dd t.
  \label{eq:Kbar-definition}
\end{equation}
We now find $\overline K$ explicitly in terms of the steady invariants.

\begin{lemma}\label{lem:explicit-averaged-first-correction-matrix}
Let $I_0$ lie on the selected active regular torus and assume that, along
this torus,
\begin{equation*}
  |\nabla I|>0,\qquad |U_{\mathrm m}|>0,\qquad V_\varphi\ne0.
\end{equation*}
The period average $\overline K$
defined in \eqref{eq:Kbar-definition} is
\begin{equation}
  \overline K=
  \begin{pmatrix}
    \overline K_{11}&\overline K_{12}\\
    \overline K_{21}&\lambda_k'(I_0)-\overline K_{11}
  \end{pmatrix},
  \label{eq:Kbar-explicit}
\end{equation}
where
\begin{align*}
  \overline K_{11}
  &=\frac{(k\cdot\Om(I_0))L_I-2mP_I}{L},
\\
  \overline K_{12}
  &=\frac1{T_0}\int_0^{T_0}
    \left(
      \frac{(k\cdot\Om(I_0))L_I-2mP_I}
           {L|U_{\mathrm m}|^2}
      -\frac{mL_I}{\rho^2|U_{\mathrm m}|^2}
      +\frac{2mV_\varphi n_\rho}
             {|\nabla I|\rho^2|U_{\mathrm m}|^2}
    \right)\dd t, \quad n_\rho:=\mathbf n\cdot\mathbf e_\rho,
\\
  \overline K_{21}
  &=(k\cdot\Om(I_0))
    \left(\frac12\partial_I(|U|^2)-P_I\right)
    -\frac{|U|^2}{L}
      \left((k\cdot\Om(I_0))L_I-2mP_I\right).
\end{align*}
In particular,
\begin{equation}
  \operatorname{tr}\overline K=\lambda_k'(I_0).
  \label{eq:Kbar-trace}
\end{equation}
\end{lemma}

\begin{proof}
We first compute $\mathsf r$ directly in the orthonormal tangent frame
$Q_\infty=(\boldsymbol\tau,\mathbf e_\varphi)$.
Since both $\boldsymbol\tau$ and $\mathbf e_\varphi$ are tangent to the invariant torus, we have 
\[\Pi_T \boldsymbol\tau=\boldsymbol\tau, \quad \Pi_T \mathbf e_\varphi=\mathbf e_\varphi.\]
Applying
\[
  \boldsymbol\tau
  =\frac{U-V_\varphi\mathbf e_\varphi}{|U_{\mathrm m}|},
  \qquad
  \eta_k\cdot U=k\cdot\Om(I_0),
  \qquad
  \eta_k\cdot\mathbf e_\varphi=\frac m\rho,
\]
we obtain
\begin{equation*}
  \mathsf r_1
  =\frac{k\cdot\Om(I_0)-mV_\varphi/\rho}
         {|\nabla I|\,|U_{\mathrm m}|},
  \qquad
  \mathsf r_2=\frac{m}{|\nabla I|\rho}.
\end{equation*}

We next compute $\mathsf c$ and $\mathsf d$.  It follows from the steady Euler equation and
the fact that $P=P(I)$ 
\begin{equation*}
  DU(x(t))U=-\nabla P=-|\nabla I|P_I\mathbf n,
  \qquad
  U\cdot DU(x(t))\mathbf n
  =\frac{|\nabla I|}{2}\partial_I(|U|^2).
\end{equation*}
Moreover, invoking \eqref{eq:DU-cylindrical-exact} we get
\begin{equation*}
  DU(x(t))\mathbf e_\varphi
  =-\frac{V_\varphi}{\rho}\mathbf e_\rho
   +\frac{V_\rho}{\rho}\mathbf e_\varphi,
  \qquad
  \mathbf n\cdot DU(x(t))\mathbf e_\varphi
  =-\frac{V_\varphi n_\rho}{\rho}.
\end{equation*}
It follows from $L=\rho V_\varphi=L(I)$ that
\[
  |\nabla I|L_I
  =\mathbf n\cdot\nabla(\rho V_\varphi)
  =V_\varphi n_\rho+\rho\,\mathbf e_\varphi\cdot DU(x(t))\mathbf n.
\]
Therefore
\begin{equation*}
  \mathbf e_\varphi\cdot DU(x(t))\mathbf n
  =\frac{|\nabla I|L_I}{\rho}
   -\frac{V_\varphi n_\rho}{\rho}.
\end{equation*}
Again since $\boldsymbol\tau$ and $\mathbf e_\varphi$ are tangent,
$Q_\infty^{\mathsf T}\Pi_T=Q_\infty^{\mathsf T}$.  Thus, we have
\begin{align*}
  \mathsf c_1
  &=\boldsymbol\tau\cdot DU(x(t))^{\mathsf T}\mathbf n
    =\frac{-|\nabla I|P_I+V_\varphi^2n_\rho/\rho}
           {|U_{\mathrm m}|},
\\
  \mathsf c_2
  &=\mathbf e_\varphi\cdot DU(x(t))^{\mathsf T}\mathbf n
    =-\frac{V_\varphi n_\rho}{\rho},
\\
  \mathsf d_1
  &=\boldsymbol\tau\cdot DU(x(t))\mathbf n
    =\frac{1}{|U_{\mathrm m}|}
      \left(
        \frac{|\nabla I|}{2}\partial_I(|U|^2)
        -\frac{|\nabla I|V_\varphi L_I}{\rho}
        +\frac{V_\varphi^2n_\rho}{\rho}
      \right),
\\
  \mathsf d_2
  &=\mathbf e_\varphi\cdot DU(x(t))\mathbf n
    =\frac{|\nabla I|L_I}{\rho}
     -\frac{V_\varphi n_\rho}{\rho},
\end{align*}
and 
\begin{equation*}
  \mathsf d-\mathsf c
  =
  \begin{pmatrix}
    \dfrac{|\nabla I|}{|U_{\mathrm m}|}
    \left(
      \dfrac12\partial_I(|U|^2)+P_I
      -\dfrac{V_\varphi L_I}{\rho}
    \right)\\[4mm]
    \dfrac{|\nabla I|L_I}{\rho}
  \end{pmatrix}.
\end{equation*}

Recall that
\begin{equation*}
  Z_\infty^{-1}=
  \begin{pmatrix}
    0&1/V_\varphi\\
    |U_{\mathrm m}|&-|U_{\mathrm m}|^2/V_\varphi
  \end{pmatrix},
  \qquad
  D=(\mathsf d-\mathsf c)\mathsf r^{\mathsf T}
    +2\mathsf r\mathsf c^{\mathsf T}.
\end{equation*}
Consequently, we deduce
\begin{align}
  D\binom{|U_{\mathrm m}|}{V_\varphi}
  &=
  \frac{k\cdot\Om(I_0)}{|\nabla I|}
    (\mathsf d-\mathsf c)
  -2|\nabla I|P_I\,\mathsf r,
  \label{eq:D-first-Z-column-direct}\\
  D\binom{1/|U_{\mathrm m}|}{0}
  &=\frac1{|U_{\mathrm m}|}
    \left(\mathsf r_1(\mathsf d-\mathsf c)
          +2\mathsf c_1\mathsf r\right),
  \label{eq:D-second-Z-column-direct}
\end{align}
where we used
\[
  \mathsf r^{\mathsf T}\binom{|U_{\mathrm m}|}{V_\varphi}
  =\frac{k\cdot\Om(I_0)}{|\nabla I|},
  \qquad
  \mathsf c^{\mathsf T}\binom{|U_{\mathrm m}|}{V_\varphi}
  =-|\nabla I|P_I.
\]

In view of 
\eqref{eq:D-first-Z-column-direct} we get
\begin{align}
  K_{11}
  &=\frac1{V_\varphi}
    \left[
      \frac{k\cdot\Om(I_0)}{|\nabla I|}
      \frac{|\nabla I|L_I}{\rho}
      -2|\nabla I|P_I
       \frac{m}{|\nabla I|\rho}
    \right]\notag\\
  &=\frac{(k\cdot\Om(I_0))L_I-2mP_I}{L},
  \label{eq:K11-explicit}
\end{align}
while it follows from \eqref{eq:D-second-Z-column-direct} that
\begin{align*}
  K_{12}
  &=\frac1{|U_{\mathrm m}|^2}
    \left[
      \frac{(k\cdot\Om(I_0))L_I}{\rho V_\varphi}
      -\frac{mL_I}{\rho^2}
      -\frac{2mP_I}{\rho V_\varphi}
      +\frac{2mV_\varphi n_\rho}{|\nabla I|\rho^2}
    \right]\notag\\
  &=\frac{(k\cdot\Om(I_0))L_I-2mP_I}
          {L|U_{\mathrm m}|^2}
    -\frac{mL_I}{\rho^2|U_{\mathrm m}|^2}
    +\frac{2mV_\varphi n_\rho}
           {|\nabla I|\rho^2|U_{\mathrm m}|^2}.
\end{align*}
Similarly, applying
\begin{align*}
  \begin{pmatrix}
    |U_{\mathrm m}|&-|U_{\mathrm m}|^2/V_\varphi
  \end{pmatrix}\mathsf r
  &=\frac1{|\nabla I|}
    \left(k\cdot\Om(I_0)-\frac{m|U|^2}{L}\right),
\\
  \begin{pmatrix}
    |U_{\mathrm m}|&-|U_{\mathrm m}|^2/V_\varphi
  \end{pmatrix}
  (\mathsf d-\mathsf c)
  &=|\nabla I|
    \left(
      \frac12\partial_I(|U|^2)+P_I
      -\frac{|U|^2L_I}{L}
    \right)
\end{align*}
to \eqref{eq:D-first-Z-column-direct} yields
\begin{align}
  K_{21}
  &=(k\cdot\Om(I_0))
    \left(
      \frac12\partial_I(|U|^2)+P_I
      -\frac{|U|^2L_I}{L}
    \right)\notag\\
  &\quad
    -2P_I
    \left(k\cdot\Om(I_0)-\frac{m|U|^2}{L}\right)\notag\\
  &=(k\cdot\Om(I_0))
    \left(\frac12\partial_I(|U|^2)-P_I\right)
    -\frac{|U|^2}{L}
      \left((k\cdot\Om(I_0))L_I-2mP_I\right).
  \label{eq:K21-explicit}
\end{align}
Averaging \eqref{eq:K11-explicit}--\eqref{eq:K21-explicit} we obtain $\overline K_{11}$, $\overline K_{12}$ and $\overline K_{21}$
in \eqref{eq:Kbar-explicit}.

It remains to determine $K_{22}$. Since $K=Z_\infty^{-1}DZ_\infty$ and
$D$ is given by \eqref{eq:D-first-correction}, we have
\begin{equation*}
  \operatorname{tr}K=\operatorname{tr}D
  =(\mathsf d-\mathsf c)\cdot\mathsf r
    +2\mathsf c\cdot\mathsf r
  =(\mathbf c+\mathbf d)\cdot\mathbf r_k.
\end{equation*}
On the other hand, since 
\begin{equation*}
  \frac{\dd}{\dd t}\nabla I(x(t))
  =-DU(x(t))^{\mathsf T}\nabla I(x(t)),
\end{equation*}
together with
$\xi=s\nabla I+\eta_k$, $\dot s=-\lambda_k'(I_0)$, and
\eqref{eq:LH-xi}, we conclude
\begin{align*}
  \dot\eta_k
  &=-DU(x(t))^{\mathsf T}\eta_k
    +\lambda_k'(I_0)|\nabla I|\,\mathbf n,\\
  \frac{\dd}{\dd t}|\nabla I|
  &=-|\nabla I|\,\mathbf n\cdot DU(x(t))\mathbf n,
  \qquad
  \dot{\mathbf n}=-\mathbf c.
\end{align*}
We then compute
\begin{align*}
  \frac{\dd}{\dd t}
    \left(\frac{\mathbf n\cdot\eta_k}{|\nabla I|}\right)
  &=\lambda_k'(I_0)
    -\frac{\mathbf c\cdot\eta_k
       +(DU(x(t))\mathbf n)\cdot\eta_k}{|\nabla I|}
    +\frac{\mathbf n\cdot\eta_k}{|\nabla I|}
       \bigl(\mathbf n\cdot DU(x(t))\mathbf n\bigr)\\
  &=\lambda_k'(I_0)-(\mathbf c+\mathbf d)\cdot\mathbf r_k.
\end{align*}
Applying
$DU(x(t))\mathbf n
=\mathbf d+(\mathbf n\cdot DU(x(t))\mathbf n)\mathbf n$ and the fact that $\mathbf c,\mathbf d$
are tangent and $\Pi_T\eta_k=|\nabla I|\mathbf r_k$, we obtain
\begin{equation}
  \operatorname{tr}K(t)=\operatorname{tr}D(t)
  =\lambda_k'(I_0)
   -\frac{\dd}{\dd t}
    \left(\frac{\mathbf n\cdot\eta_k}{|\nabla I|}\right).
  \label{eq:direct-first-correction-trace}
\end{equation}
The remaining entry is therefore
\begin{equation}
  K_{22}(t)
  =\lambda_k'(I_0)
   -\frac{\dd}{\dd t}
    \left(\frac{\mathbf n\cdot\eta_k}{|\nabla I|}\right)
   -K_{11}.
  \label{eq:K22-pointwise-direct}
\end{equation}
Note $(\mathbf n\cdot\eta_k)/|\nabla I|$ is $T_0$-periodic.
After averaging \eqref{eq:direct-first-correction-trace} over one period, we infer
\[
  \operatorname{tr}\overline K
  =\frac1{T_0}\int_0^{T_0}\operatorname{tr}D(t)\,\dd t
  =\lambda_k'(I_0),
\]
which implies \eqref{eq:Kbar-trace} and hence \eqref{eq:Kbar-explicit} follows as well.  
\end{proof}

In particular,
\begin{equation}
  \det\overline K
  =\overline K_{11}\bigl(\lambda_k'(I_0)-\overline K_{11}\bigr)
   -\overline K_{12}\overline K_{21}.
  \label{eq:det-Kbar-general}
\end{equation}
For the trace-free normalization in
\eqref{eq:exact-trace-free-amplitude-system}, the corresponding averaged
matrix is
\begin{equation*}
  \overline K_{\rm tf}
  =\overline K-\frac{\lambda_k'(I_0)}2\mathrm{Id},
  \qquad
  \det\overline K_{\rm tf}
  =\det\overline K-\frac{\bigl(\lambda_k'(I_0)\bigr)^2}{4}.
\end{equation*}

For the axisymmetric mode $k=(n,0)$, \eqref{eq:det-Kbar-general} reduces to
\begin{equation}
  \begin{aligned}
    \det\overline K
    ={}&n^2\Omega_1^2\frac{L_I}{L}\Biggl[
      \frac{\Omega_1'}{\Omega_1}-\frac{L_I}{L}
    -\left\langle\frac1{|U_{\mathrm m}|^2}\right\rangle_{T_0}
      \left(
        \frac12\partial_I(|U|^2)-P_I
        -|U|^2\frac{L_I}{L}
      \right)
    \Biggr].
  \end{aligned}
  \label{eq:det-Kbar-axisymmetric-reduction}
\end{equation}
Here $\langle\cdot\rangle_{T_0}$ denotes the time average over one
meridional period.

\begin{proposition}\label{prop:negative-averaged-determinant}
Fix $k=(n,0)$ with $n\ne0$ and a plateau height $\chi_*>0$.  Consider the
family of localized Gavrilov flows obtained from
\eqref{eq:explicit-pressure-cutoff} as $\iota\downarrow0$.  Uniformly for
$I\in[\iota,2\iota]$,
\begin{equation}
  \det\overline K
  =-\frac{3n^2\chi_*^2}{2I}+O(n^2\chi_*^2).
  \label{eq:det-Kbar-small-action}
\end{equation}
In particular, $\det\overline K<0$ throughout the band $I\in[\iota,2\iota]$ for all sufficiently small $\iota$.
\end{proposition}
\begin{proof}
Recall from Appendix~\ref{app:Baldi-formulas} that
$\mathfrak c=4\Kcal(I)$ and
$\mathfrak c_I=\dd\mathfrak c/\dd I$, and 
$\gamma^{\rm mer}_{\mathfrak c}(\vartheta)$ is Baldi's meridional level
graph.  Denote
\begin{align}
  \rho_{\mathfrak c}(\vartheta)
  &:=1+\sqrt{2\gamma^{\rm mer}_{\mathfrak c}(\vartheta)}
    \sin\vartheta,
  \notag\\
  \mathcal T(\mathfrak c)
  &:=\int_0^{2\pi}\partial_{\mathfrak c}
    \gamma^{\rm mer}_{\mathfrak c}(\vartheta)\,\dd\vartheta,
  \notag\\
  \mathcal M(\mathfrak c)
  &:=\frac1{\mathcal T(\mathfrak c)}
  \int_0^{2\pi}
    \frac{\partial_{\mathfrak c}
      \gamma^{\rm mer}_{\mathfrak c}(\vartheta)}
    {12\mathfrak c-
      \mathcal H(\mathfrak c)/\rho_{\mathfrak c}(\vartheta)^2}
    \,\dd\vartheta.
  \label{eq:M-level-quadrature}
\end{align}
The angle-straightening formula
\eqref{eq:Baldi-appendix-angle-straightening} implies
\begin{equation}\label{measure}
  \frac{\dd\sigma}{2\pi}
  =\frac{\partial_{\mathfrak c}
    \gamma^{\rm mer}_{\mathfrak c}(\vartheta)}
    {\mathcal T(\mathfrak c)}\,\dd\vartheta.
\end{equation}
Along the fixed torus, $\dot\sigma=\Omega_1(I)$ and $T_0=2\pi/\Omega_1(I)$; hence $dt/T_0=d\sigma/(2\pi)$. 

On the cutoff plateau, Gavrilov's first integrals give
\begin{equation*}
  \Omega_1=\frac{\chi_*}{4}\mathfrak c_I,
  \qquad
  L=\frac{\chi_*}{4}\sqrt{\mathcal H(\mathfrak c)},
  \qquad
  |U|^2=\frac{3\chi_*^2}{4}\mathfrak c,
  \qquad
  P_I=\frac{\chi_*^2}{4}\mathfrak c_I,
\end{equation*}
and
\begin{equation*}
  |U_{\mathrm m}|^2=\frac{\chi_*^2}{16}
  \left(12\mathfrak c-
    \frac{\mathcal H(\mathfrak c)}
         {\rho_{\mathfrak c}(\vartheta)^2}\right),
  \qquad
  \left\langle\frac1{|U_{\mathrm m}|^2}\right\rangle_{T_0}
  =\frac{16}{\chi_*^2}\mathcal M(\mathfrak c),
\end{equation*}
where the last equality follows from \eqref{measure}.  Since $\chi_*$ is constant on the plateau, we deduce from \eqref{eq:Baldi-appendix-action-period-identity} that
\begin{equation*}
  \frac{\Omega_1'}{\Omega_1}
  =\frac{\dd}{\dd I}\log\mathfrak c_I
  =-\mathfrak c_I
    \frac{\mathcal T_{\mathfrak c}(\mathfrak c)}
         {\mathcal T(\mathfrak c)},
  \qquad
  \frac{L_I}{L}
  =\mathfrak c_I
    \frac{\mathcal H_{\mathfrak c}(\mathfrak c)}
         {2\mathcal H(\mathfrak c)}.
\end{equation*}
On the other hand we have
\begin{equation}
  \frac12\partial_I(|U|^2)-P_I-|U|^2\frac{L_I}{L}
  =\frac{\chi_*^2\mathfrak c_I}{8}
  \left(1-3\mathfrak c
    \frac{\mathcal H_{\mathfrak c}(\mathfrak c)}
         {\mathcal H(\mathfrak c)}\right).
  \label{eq:Bernoulli-combination-in-level}
\end{equation}
Consequently, it follows from \eqref{eq:det-Kbar-axisymmetric-reduction} that
\begin{equation}
  \det\overline K
  =\frac{n^2\chi_*^2\mathfrak c_I^4}{16}
   \frac{\mathcal H_{\mathfrak c}(\mathfrak c)}
        {2\mathcal H(\mathfrak c)}\,
   \mathcal G(\mathfrak c),
  \label{eq:det-Kbar-exact-quadrature}
\end{equation}
with
\begin{equation}
  \begin{aligned}
    \mathcal G(\mathfrak c):={}&
    -\frac{\mathcal T_{\mathfrak c}(\mathfrak c)}
           {\mathcal T(\mathfrak c)}
    -\frac{\mathcal H_{\mathfrak c}(\mathfrak c)}
           {2\mathcal H(\mathfrak c)}
    -2\left(1-3\mathfrak c
      \frac{\mathcal H_{\mathfrak c}(\mathfrak c)}
           {\mathcal H(\mathfrak c)}\right)
      \mathcal M(\mathfrak c).
  \end{aligned}
  \label{eq:G-level-exact}
\end{equation}

The level-coefficient expansion
\eqref{eq:Baldi-appendix-level-coefficients} and the physical-coordinate
reconstruction in Appendix~\ref{app:Baldi-formulas} provide the following higher-order
terms:
\begin{align*}
  \partial_{\mathfrak c}
    \gamma^{\rm mer}_{\mathfrak c}(\vartheta)
  &=\frac14+\frac32\mathcal Q_3^{\rm mer}(\vartheta)
      \sqrt{\mathfrak c}
    +2\mathcal Q_4^{\rm mer}(\vartheta)\mathfrak c
    +O(\mathfrak c^{3/2}),
\\
  \rho_{\mathfrak c}^{-2}
  &=1-\sqrt2\sin\vartheta\sqrt{\mathfrak c}
    +\left(1-\frac98\cos2\vartheta
      +\frac18\cos4\vartheta\right)\mathfrak c
      +O(\mathfrak c^{3/2}),
\end{align*}
where
\begin{equation*}
  \mathcal Q_3^{\rm mer}(\vartheta)
  =-\frac{2\sin\vartheta-\sin3\vartheta}{8\sqrt2},
  \qquad
  \frac1{2\pi}\int_0^{2\pi}
    \mathcal Q_4^{\rm mer}(\vartheta)\,\dd\vartheta=0.
\end{equation*}
Applying the terms above, the denominator in \eqref{eq:M-level-quadrature} has the expansion
\begin{align*}
  12\mathfrak c
  -\mathcal H(\mathfrak c)\rho_{\mathfrak c}^{-2}
  ={}&8\mathfrak c
  +4\sqrt2\sin\vartheta\,\mathfrak c^{3/2}\\
  &+\left(
      \frac{13}{2}+\frac92\cos2\vartheta
      -\frac12\cos4\vartheta
    \right)\mathfrak c^2
  +O(\mathfrak c^{5/2}),
\end{align*}
and hence
\begin{align*}
  &\frac{1}{12\mathfrak c
    -\mathcal H(\mathfrak c)\rho_{\mathfrak c}^{-2}}
  =\frac1{8\mathfrak c}\Biggl[
    1-\frac{\sqrt2}{2}\sin\vartheta\sqrt{\mathfrak c}\\
  &\hspace{42mm}
    +\left\{
      \frac12\sin^2\vartheta-\frac{13}{16}
      -\frac9{16}\cos2\vartheta
      +\frac1{16}\cos4\vartheta
    \right\}\mathfrak c
    +O(\mathfrak c^{3/2})
  \Biggr].
\end{align*}
Straightforward computation shows 
\begin{equation*}
  \frac1{2\pi}\int_0^{2\pi}
  \frac{\partial_{\mathfrak c}
    \gamma^{\rm mer}_{\mathfrak c}(\vartheta)}
  {12\mathfrak c-
    \mathcal H(\mathfrak c)/\rho_{\mathfrak c}(\vartheta)^2}
  \,\dd\vartheta
  =\frac1{32\mathfrak c}-\frac3{512}+O(\mathfrak c).
\end{equation*}
Since
$\mathcal T(\mathfrak c)=\frac\pi2+O(\mathfrak c^2)$, substitution in
\eqref{eq:M-level-quadrature} yields
\begin{equation}
  \mathcal M(\mathfrak c)
  =\frac1{8\mathfrak c}-\frac3{128}+O(\mathfrak c).
  \label{eq:M-level-sharp-expansion}
\end{equation}
Inserting \eqref{eq:Baldi-appendix-H-expansion},
\eqref{eq:Baldi-appendix-T-expansion}, and
\eqref{eq:M-level-sharp-expansion} into \eqref{eq:G-level-exact}, we have
\begin{equation}\notag
  \frac{\mathcal H_{\mathfrak c}}{\mathcal H}
  =\frac1{\mathfrak c}-\frac{21}{8}+O(\mathfrak c), \quad
  \frac{\mathcal T_{\mathfrak c}}{\mathcal T}=O(\mathfrak c), \quad
  1-3\mathfrak c\frac{\mathcal H_{\mathfrak c}}{\mathcal H}
  =-2+\frac{63}{8}\mathfrak c+O(\mathfrak c^2),
\end{equation}
and 
\begin{align*}
  -2\left(1-3\mathfrak c
    \frac{\mathcal H_{\mathfrak c}}{\mathcal H}\right)
    \mathcal M(\mathfrak c)
  &=-2\left(-2+\frac{63}{8}\mathfrak c+O(\mathfrak c^2)\right)
    \left(\frac1{8\mathfrak c}-\frac3{128}+O(\mathfrak c)\right)\\
  &=\frac1{2\mathfrak c}
    -\frac3{32}-\frac{63}{32}+O(\mathfrak c)\\
  &=\frac1{2\mathfrak c}-\frac{33}{16}+O(\mathfrak c).
\end{align*}
Combining the three terms in \eqref{eq:G-level-exact} yields
\begin{equation}
  \mathcal G(\mathfrak c)
  =\left(-\frac1{2\mathfrak c}+\frac{21}{16}\right)
    +\left(\frac1{2\mathfrak c}-\frac{33}{16}\right)
    +O(\mathfrak c)
  =-\frac34+O(\mathfrak c).
  \label{eq:G-level-sharp-expansion}
\end{equation}
Finally, the action normalization
\eqref{eq:Baldi-appendix-action-normalization} gives
$\mathfrak c_I=4+O(\mathfrak c^2)$ and
$\mathfrak c=4I+O(I^3)$.  Equation
\eqref{eq:det-Kbar-exact-quadrature} therefore implies
\begin{equation*}
  \det\overline K
  =-\frac{6n^2\chi_*^2}{\mathfrak c}+O(n^2\chi_*^2)
  =-\frac{3n^2\chi_*^2}{2I}+O(n^2\chi_*^2).
\end{equation*}
This completes the proof of the proposition.
\end{proof}

\subsection{The critical torus with Floquet hyperbolicity and exponential amplitude growth}
\label{subsubsec:critical-torus-exponential-cocycle}

While the mechanism above uses $\Omega _1'\ne0$, we proceed to explore the mechanism at a critical torus where $\Omega _1'=0$. For every active action $I$ and every axisymmetric mode
$k=(n,0)$, let $D(t;I)$ be the matrix in
\eqref{eq:D-first-correction}, computed along the orbit on the torus labeled
by $I$, and let $Z_\infty(t;I)$ be the corresponding periodic fundamental
matrix from \eqref{eq:limiting-fundamental-Z} and 
\begin{equation*}
  K(t;I):=Z_\infty(t;I)^{-1}D(t;I)Z_\infty(t;I).
\end{equation*}
The matrix $K(t;I)$ is periodic in the cylindrical coordinate and smooth in term of $I$.
As before, define
\begin{equation}
  \overline K(I)
  :=\frac1{T_0(I)}\int_0^{T_0(I)}K(t;I)\,\dd t.
  \label{eq:normal-Kbar-definition}
\end{equation}
This is exactly the matrix $\overline K$ in \eqref{eq:Kbar-definition}, regardless of
whether $\lambda_k'(I)=0$, and
Lemma~\ref{lem:explicit-averaged-first-correction-matrix} applies on each
active regular torus.  

\begin{proposition}[A hyperbolic averaged matrix on a critical torus]
\label{prop:critical-torus-negative-determinant}
For all sufficiently small $\iota$, there is
\begin{equation}
  I_{\rm c}\in(9\iota/4,5\iota/2)
  \label{eq:critical-torus-location}
\end{equation}
such that
\begin{equation}
  \Omega _1'(I_{\rm c})=0,
  \qquad \Omega _1(I_{\rm c})>0.
  \label{eq:critical-torus-zero-twist}
\end{equation}
For every nonzero axisymmetric mode $k=(n,0)$, the matrix
$\overline K$ in \eqref{eq:normal-Kbar-definition}, evaluated on this torus,
satisfies
\begin{equation}
  \operatorname{tr}\overline K(I_{\rm c})=0
  \label{eq:critical-torus-Kbar-trace}
\end{equation}
and
\begin{equation}
  \det\overline K(I_{\rm c})
  =-\frac{6n^2\chi_\iota(\Kcal(I_{\rm c}))^2}
          {\mathfrak c(I_{\rm c})}
   +O\!\left(n^2\chi_\iota(\Kcal(I_{\rm c}))^2\right)<0.
  \label{eq:critical-torus-Kbar-determinant}
\end{equation}
The implicit constant is independent of the value of the positive cutoff
height $\chi_\iota(\Kcal(I_{\rm c}))$.
\end{proposition}

\begin{proof}
Note $\lambda_k'(I_{\rm c})=0$ for every axisymmetric mode
$k=(n,0)$.
The function $\Omega _1$ is positive in the open active band and vanishes
at both flat edges.  On the plateau through $I=9\iota/4$,
$\Omega _1'=\chi_*\Kcal''>0$ by
\eqref{eq:K-second-derivative-lower-bound}.  Hence $\Omega _1$ attains an
interior maximum at some $I_{\rm c}$ satisfying
\eqref{eq:critical-torus-location}--
\eqref{eq:critical-torus-zero-twist}.

The trace identity \eqref{eq:direct-first-correction-trace} applies
on this torus without a nonzero-twist hypothesis.  Since
$(\mathbf n\cdot\eta_k)/|\nabla I|$ is periodic, we obtain by averaging it 
\begin{equation*}
  \operatorname{tr}\overline K(I_{\rm c})
  =\lambda_k'(I_{\rm c})
  =n\Omega_1'(I_{\rm c})=0.
\end{equation*}
which proves \eqref{eq:critical-torus-Kbar-trace}.

The localized torus invariants evaluated at $I=I_{\rm c}$ are
\begin{equation}  \label{eq:critical-torus-localized-invariants}
\begin{split}
  \Omega _1(I_{\rm c})
  &=\frac14\chi_\iota(\Kcal(I_{\rm c}))
    \mathfrak c_I(I_{\rm c}), \quad
  L(I_{\rm c})
  =\frac14\chi_\iota(\Kcal(I_{\rm c}))
    \sqrt{\mathcal H(\mathfrak c(I_{\rm c}))},\\
  |U|^2(I_{\rm c})
  &=\frac34\chi_\iota(\Kcal(I_{\rm c}))^2
    \mathfrak c(I_{\rm c}),\quad
  P_I(I_{\rm c})
  =\frac14\chi_\iota(\Kcal(I_{\rm c}))^2
    \mathfrak c_I(I_{\rm c}),
    \end{split}
\end{equation}
and the normalized orbit average is
\begin{equation}
  \left\langle\frac1{|U_{\mathrm m}|^2}\right\rangle_{T_0}
  =\frac{16}{\chi_\iota(\Kcal(I_{\rm c}))^2}
    \mathcal M(\mathfrak c(I_{\rm c})).
  \label{eq:critical-torus-meridional-speed-average}
\end{equation}
Since $\Omega _1'(I_{\rm c})=0$ and $\Omega_1(I_{\rm c})>0$, we have
\begin{equation}
  \left.
  \frac{\dd}{\dd I}\log\chi_\iota(\Kcal(I))
  \right|_{I=I_{\rm c}}
  =-\frac{\mathfrak c_{II}(I_{\rm c})}
          {\mathfrak c_I(I_{\rm c})}.
  \label{eq:critical-torus-log-cutoff-derivative}
\end{equation}
It follows from \eqref{eq:critical-torus-localized-invariants} that
\begin{equation*}
  \left.
  \left(
    \frac12\partial_I(|U|^2)-P_I-|U|^2\frac{L_I}{L}
  \right)\right|_{I=I_{\rm c}}
  =\frac{\chi_\iota(\Kcal(I_{\rm c}))^2
          \mathfrak c_I(I_{\rm c})}{8}
    \left(
      1-3\mathfrak c(I_{\rm c})
      \frac{\mathcal H_{\mathfrak c}(\mathfrak c(I_{\rm c}))}
           {\mathcal H(\mathfrak c(I_{\rm c}))}
    \right),
\end{equation*}
which combined with \eqref{eq:critical-torus-log-cutoff-derivative} and the action identity 
\[ \mathfrak c_I(I)\mathcal T(\mathfrak c(I))=2\pi\]
implies
\begin{equation}
  \left.\frac{L_I}{L}\right|_{I=I_{\rm c}}
  =\mathfrak c_I(I_{\rm c})\left(
    \frac{\mathcal T_{\mathfrak c}(\mathfrak c(I_{\rm c}))}
         {\mathcal T(\mathfrak c(I_{\rm c}))}
    +\frac{\mathcal H_{\mathfrak c}(\mathfrak c(I_{\rm c}))}
           {2\mathcal H(\mathfrak c(I_{\rm c}))}
  \right).
  \label{eq:critical-torus-L-log-derivative}
\end{equation}

  At $I=I_{\rm c}$, the wave-covector formula
  \eqref{eq:wave-covector} yields
  \[
    q(t)=(n,0,\ell),
    \qquad
    \xi_\ell(t)=\ell\nabla I(x(t))+\eta_k(t),
  \]
since $k\cdot\Om'(I_{\rm c})=n\Omega_1'(I_{\rm c})=0$.

Denote
 $T_{0,c} =\frac{2\pi}{\Omega_1(I_{\rm c})}.$
  For every sufficiently large
  $\ell$, applying Lemmas~\ref{lem:first-corrected-tangent-amplitude} and
  \ref{lem:post-limiting-Floquet-reduction} with $s(t)=\ell$,
  we obtain the periodic amplitude system
  \begin{equation}
    \dot z_{\mathrm F}=\left(\frac1\ell K(t;I_{\rm c})
      +\frac1{\ell^2}R_\ell(t)\right)z_{\mathrm F},
    \label{eq:critical-direct-normal-expansion}
  \end{equation}
  where $R_\ell$ is $ T_{0,c}$-periodic and for some $\ell_0$
  \begin{equation}
    \sup_{\ell\ge\ell_0}\|R_\ell\|_{C^1_t([0, T_{0,c}])}<\infty.
    \label{eq:critical-direct-normal-remainder}
  \end{equation}

  We next compute the period map. Let $\Phi_\ell(t)$ be the fundamental matrix of
  \eqref{eq:critical-direct-normal-expansion}, normalized by
  $\Phi_\ell(0)=\mathrm{Id}$.
  The Volterra equation is
  \begin{equation*}
    \Phi_\ell(t)=\mathrm{Id}
    +\int_0^t\left(
      \frac1\ell K(\tau;I_{\rm c})
      +\frac1{\ell^2}R_\ell(\tau)
    \right)\Phi_\ell(\tau)\,\dd\tau.
  \end{equation*}
  The uniform coefficient bounds in
  \eqref{eq:critical-direct-normal-remainder} and Gronwall's inequality imply
  \begin{equation*}
    \Phi_\ell(t)=\mathrm{Id}+O(\ell^{-1}),
    \qquad 0\le t\le  T_{0,c}.
  \end{equation*}
Thus it follows from the Volterra equation that
  \begin{equation}
  \begin{split}
    \Phi_\ell( T_{0,c})
    &=\mathrm{Id}
      +\frac1\ell\int_0^{ T_{0,c}}
        K(t;I_{\rm c})\,\dd t+O(\ell^{-2})\\
    &=\mathrm{Id}+\frac{ T_{0,c}}\ell
      \overline K(I_{\rm c})+O(\ell^{-2}),
  \end{split}
  \label{eq:critical-torus-monodromy-expansion}
  \end{equation}
  where the last equality used
  \eqref{eq:normal-Kbar-definition}.

  We now verify that the reduced period map also has determinant one.  For $p\in\nabla I(x(t))^\perp$ and with $s=\ell$, denote the graph map in
  \eqref{eq:amplitude-graph-map} by
  \begin{equation*}
    \mathfrak J_\ell(t)p
    :=p-\frac{p\cdot\eta_k(t)}
      {\ell|\nabla I(x(t))|
       +\mathbf n(t)\cdot\eta_k(t)}\,\mathbf n(t).
  \end{equation*}
  The transformations
  \[
    z_{\mathrm F}\longmapsto
    z_{\mathrm{tan}}=Z_\infty(t)z_{\mathrm F}
    \longmapsto p=Q_\infty(t)z_{\mathrm{tan}}
    \longmapsto a=\mathfrak J_\ell(t)p
  \]
  reconstruct the physical amplitude. Let $\mathscr O_{\rm orb}$ be the orientation-preserving rotation about
  the symmetry axis accumulated by the orbit over one meridional period.
  Periodicity in the cylindrical identification implies
  \begin{equation*}
    \xi_\ell(T_{0,c})
      =\mathscr O_{\rm orb}\xi_\ell(0),
    \qquad
   \mathfrak J_\ell(T_{0,c})Q_\infty(T_{0,c})Z_\infty(T_{0,c})
      =\mathscr O_{\rm orb} \mathfrak J_\ell(0)Q_\infty(0)Z_\infty(0).
  \end{equation*}
  Then $a(t)=\mathfrak J_\ell(t)Q_\infty(t)Z_\infty(t)z_{\mathrm F}(t)$ implies
  \begin{equation}
    \mathscr O_{\rm orb}^{-1}\mathcal U_a(T_{0,c})
    \mathfrak J_\ell(0)Q_\infty(0)Z_\infty(0)
    =\mathfrak J_\ell(0)Q_\infty(0)Z_\infty(0)\Phi_\ell(T_{0,c})
    \label{eq:critical-physical-reduced-monodromy-intertwining}
  \end{equation}
with $\mathcal U_a(T_{0,c})$ representing the physical amplitude map along this bicharacteristic.

  Since $\mathscr O_{\rm orb}$ is an orientation-preserving isometry and
  the critical covector has the same length after one period, we infer from the
  amplitude-area identity \eqref{eq:LH-amplitude-area-Jacobian} and
  \eqref{eq:critical-physical-reduced-monodromy-intertwining} that
  \[
    \bigl|\det\Phi_\ell(T_{0,c})\bigr|
    =\bigl|\det(\mathscr O_{\rm orb}^{-1}
      \mathcal U_a(T_{0,c}))\bigr|
    =\frac{|\xi_\ell(0)|}{|\xi_\ell(T_{0,c})|}=1.
  \]
  Since \eqref{eq:critical-torus-monodromy-expansion} shows that
  $\Phi_\ell(T_{0,c})=\mathrm{Id}+O(\ell^{-1})$,  $\det\Phi_\ell(T_{0,c})>0$ for all sufficiently large $\ell$; hence $\det\Phi_\ell(T_{0,c})=1$.

  To compute $\det\overline K(I_{\rm c})$, we apply
  Lemma~\ref{lem:explicit-averaged-first-correction-matrix} and the trace
  identity \eqref{eq:critical-torus-Kbar-trace}.  Thus
  $\overline K_{22}(I_{\rm c})=-\overline K_{11}(I_{\rm c})$, and
  \eqref{eq:det-Kbar-axisymmetric-reduction} applies with
  $\lambda_k'(I_{\rm c})=0$.  It follows from
  \eqref{eq:critical-torus-meridional-speed-average}--
  \eqref{eq:critical-torus-L-log-derivative} that
  \begin{equation}
    \det\overline K(I_{\rm c})
    =\frac{n^2\chi_\iota(\Kcal(I_{\rm c}))^2
      \mathfrak c_I(I_{\rm c})^4}{16}
     \left(
      \frac{\mathcal T_{\mathfrak c}(\mathfrak c(I_{\rm c}))}
           {\mathcal T(\mathfrak c(I_{\rm c}))}
      +\frac{\mathcal H_{\mathfrak c}(\mathfrak c(I_{\rm c}))}
             {2\mathcal H(\mathfrak c(I_{\rm c}))}
     \right)\mathcal G(\mathfrak c(I_{\rm c})).
    \label{eq:critical-torus-Kbar-exact}
  \end{equation}
  Equations \eqref{eq:Baldi-appendix-H-expansion},
  \eqref{eq:Baldi-appendix-T-expansion}, and
  \eqref{eq:G-level-sharp-expansion} imply
  \begin{equation*}
  \begin{split}
    \frac{\mathcal T_{\mathfrak c}(\mathfrak c(I_{\rm c}))}
         {\mathcal T(\mathfrak c(I_{\rm c}))}
     +\frac{\mathcal H_{\mathfrak c}(\mathfrak c(I_{\rm c}))}
            {2\mathcal H(\mathfrak c(I_{\rm c}))}
    &=\frac1{2\mathfrak c(I_{\rm c})}+O(1),\\
    \mathcal G(\mathfrak c(I_{\rm c}))
    &=-\frac34+O(\mathfrak c(I_{\rm c})),\\
    \mathfrak c_I(I_{\rm c})
    &=4+O(\mathfrak c(I_{\rm c})^2).
  \end{split}
  \end{equation*}
  Since $\mathfrak c(I_{\rm c})=O(\iota)$,
  \eqref{eq:critical-torus-Kbar-determinant} follows from
  \eqref{eq:critical-torus-Kbar-exact}.  This completes the proof.
\end{proof}


Proposition~\ref{prop:critical-torus-negative-determinant} identifies a critical torus on which the averaged
first correction is hyperbolic.  We now show the existence of an unstable Floquet multiplier for the periodic amplitude system.

\begin{proposition}
\label{prop:critical-torus-exponential-LH}
Let $I_{\rm c}$ be from
Proposition~\ref{prop:critical-torus-negative-determinant}.
For every fixed $n\ne0$ and all sufficiently large positive $\ell$, the
axisymmetric ray with phase
\begin{equation*}
  S=n\sigma+\ell I
\end{equation*}
on $I=I_{\rm c}$ admits a nonzero real initial polarization
$a_\ell(0)\in\xi_\ell(0)^\perp$ whose corresponding physical velocity
amplitude satisfies
\begin{equation}
  |a_\ell(t)|\ge c_\ell e^{\sigma_\ell t}|a_\ell(0)|,
  \qquad t\ge0,
  \qquad
  \sigma_\ell
  =\frac{\sqrt{-\det\overline K(I_{\rm c})}}{\ell}
    +O(\ell^{-2})>0
  \label{eq:critical-torus-exponential-amplitude}
\end{equation}
for a constant $c_\ell>0$ independent of $t$.
\end{proposition}

\begin{proof}
We note the reduced system \eqref{eq:critical-direct-normal-expansion} has real and $T_{0,c}$-periodic coefficients.
Proposition~\ref{prop:critical-torus-negative-determinant} shows that the
eigenvalues of $\overline K(I_{\rm c})$ are
$\pm\sqrt{-\det\overline K(I_{\rm c})}$.  Standard perturbation theory
for periodic linear systems \cite{CoddingtonLevinson1955}
applied to \eqref{eq:critical-torus-monodromy-expansion} produces Floquet multipliers
\begin{equation*}
  \rho_{\ell,\pm}
  =1\pm\frac{T_{0,c}}\ell
    \sqrt{-\det\overline K(I_{\rm c})}
    +O(\ell^{-2}).
\end{equation*}
Thus $\rho_{\ell,+}>1$ for large $\ell$.  Define its Floquet
exponent by
\begin{equation*}
  \sigma_\ell:=\frac1{T_{0,c}}\log\rho_{\ell,+}
  =\frac{\sqrt{-\det\overline K(I_{\rm c})}}\ell
    +O(\ell^{-2})>0.
\end{equation*}
Since the real period map $\Phi_\ell(T_{0,c})$ has the real simple
eigenvalue $\rho_{\ell,+}$, we choose a real eigenvector
$v_{\ell,+}\ne0$ and let $z_{\mathrm F}(t)$ be the solution with
$z_{\mathrm F}(0)=v_{\ell,+}$.  For $t=qT_{0,c}+\tau$ with $q\in\mathbb N_0$ and
$0\le\tau<T_{0,c}$, periodicity implies
\begin{equation*}
  z_{\mathrm F}(t)
  =\Phi_\ell(\tau)\Phi_\ell(T_{0,c})^qv_{\ell,+}
  =\rho_{\ell,+}^q\Phi_\ell(\tau)v_{\ell,+}.
\end{equation*}
In view of $\rho_{\ell,+}^q=e^{\sigma_\ell(t-\tau)}$, it follows that
\begin{equation*}
  |z_{\mathrm F}(t)|\ge c'_{\ell}e^{\sigma_\ell t}
    |v_{\ell,+}|,
  \qquad t\ge0,
\end{equation*}
for some $c'_{\ell}>0$.

For sufficiently large $\ell$, the real reconstruction map $a_\ell(t)=\mathfrak J_\ell(t)Q_\infty(t)Z_\infty(t)z_{\mathrm F}(t)$ is invertible from $\mathbb R^2$ onto $\xi_\ell(t)^\perp$, and its norm and inverse norm
are bounded uniformly in $t$. Hence $a_\ell(0)$ is a nonzero real transverse polarization, and
the estimate above yields
\eqref{eq:critical-torus-exponential-amplitude} after adjusting the constant $c_\ell$.
\end{proof}

\section{Instability of the linearized Euler equation}
\label{sec-lift}

We now lift the exponential growth of the critical-torus velocity cocycle to
the linearized Euler evolution.

\subsection{Evolution of the linearized Euler equation}

Fix one of the smooth localized equilibria $U=U_\iota$ and, until the proof
of the main theorem, suppress the subscript $\iota$.  Let $G(t)$ denote the
linearized Euler evolution group on the complexified
energy space 
\[L^2_{\rm div}(\mathbb R^3):=
\left\{
v\in L^2(\mathbb R^3;\mathbb C^3):
\nabla\cdot v=0
\ \text{in }\mathcal D'(\mathbb R^3)
\right\}.\]  
Since $U$ is smooth and compactly supported, $G(t)$ acts boundedly on the
divergence-free subspace of $H^s(\R^3;\mathbb C^3)$ for every fixed $t$
and every $s\ge0$.
Recall that the essential norm of a
bounded operator is
\begin{equation*}
  \|\mathscr T\|_{\rm ess}
  :=\inf_{\mathcal C\,\mathrm{compact}}
    \|\mathscr T-\mathcal C\|.
\end{equation*}
For rotations $\mathscr O_\alpha$ about the symmetry axis, denote
\begin{equation*}
  H^s_{\rm axi}
  :=\{v\in H^s(\R^3;\mathbb C^3):\Div v=0,
    (\mathscr O_\alpha)_*v=v\ \text{for every }\alpha\in\T\}.
\end{equation*}
This is a closed subspace of $H^s$ and is invariant under $G(t)$ since the linearized Euler evolution commute with rotations due to the axisymmetry of $U$.

Let $X_t$ be the physical Lagrangian flow of $U$,
\begin{equation*}
  \partial_tX_t(x)=U(X_t(x)),\qquad X_0(x)=x.
\end{equation*}
For $\xi\ne0$, let $\mathscr U(t;x,\xi)$ denote the solution operator of
the bicharacteristic--amplitude equation above the bicharacteristic with
physical initial point and covector $(x,\xi)$.  Its transported covector is
\begin{equation*}
  \xi_t(x,\xi):=DX_t(x)^{-\mathsf T}\xi.
\end{equation*}

\begin{lemma}\label{lem:fixed-time-Euler-geometric-optics-transfer}
Let $A\in C_0^\infty(\R^3;\mathbb C^3)$, and let $\phi$ be smooth on a
neighborhood of $\operatorname{supp}A$ with
\begin{equation*}
  \dd\phi\ne0,
  \qquad A(x)\cdot\dd\phi(x)=0
  \quad\text{on }\operatorname{supp}A.
\end{equation*}
For $j\to+\infty$, denote
\begin{equation*}
  f_j:=\PP\bigl[Ae^{ij\phi}\bigr].
\end{equation*}
Then, for every fixed $T\ge0$ and $s\ge0$, we have
\begin{align}
  j^{-s}\langle D\rangle^sf_j(x)
  &=e^{ij\phi(x)}|\dd\phi(x)|^sA(x)+o_{L^2}(1),
  \label{eq:fixed-time-Euler-WKB-input}\\
  j^{-s}\bigl(\langle D\rangle^sG(T)f_j\bigr)(X_T(x))
  &=e^{ij\phi(x)}|\xi_T(x,\dd\phi(x))|^s
    \mathscr U(T;x,\dd\phi(x))A(x)+o_{L^2}(1)
  \label{eq:fixed-time-Euler-WKB-output}
\end{align}
with $\langle D\rangle^s=(1-\Delta)^{s/2}$.
The terms $o_{L^2}(1)$ may depend on $A$, $\phi$, $T$, and $s$.
\end{lemma}

\begin{proof}
For a divergence-free perturbation $v$, we have
\begin{equation*}
  \partial_tv=-\PP\bigl[(U\cdot\nabla)v+DU\,v\bigr],
\end{equation*}
which can be written as
\begin{equation}
  \partial_tv=-(U\cdot\nabla)v+DU\,v-2\PP(DU\,v),
  \label{eq:Euler-advective-pseudodifferential-form}
\end{equation}
since $(U\cdot\nabla)v-DU\,v$ is divergence free.  
Note the principal symbol of the last two terms in \eqref{eq:Euler-advective-pseudodifferential-form} is
\begin{equation}
  \left(2\frac{\xi\otimes\xi}{|\xi|^2}-\mathrm{Id}\right)DU(x),
  \label{eq:Euler-velocity-principal-symbol}
\end{equation}
which is precisely the right-hand side of \eqref{eq:LH-a} under the constraint $a\cdot\xi=0$.  Thus $\mathscr U$ is the principal cocycle associated with \eqref{eq:Euler-advective-pseudodifferential-form}.

As shown in \cite[Section~3.5]{Shvydkoy2006}, for $s\geq 0$, conjugation by the operator $\langle D\rangle^s$ produces the multiplier to the principal cocycle by
\begin{equation*}
  \left(\frac{|DX_T(x)^{-\mathsf T}\xi|}{|\xi|}\right)^s.
\end{equation*}
We obtain \eqref{eq:fixed-time-Euler-WKB-input} and \eqref{eq:fixed-time-Euler-WKB-output} by applying a similar analysis to the conjugated evolution.
\end{proof}

Applying Lemma \ref{lem:fixed-time-Euler-geometric-optics-transfer}, we first prove a general $L^2$ lower bound.  Its
symmetry-preserving $H^s$ refinement follows afterward.

\begin{proposition}\label{prop:exact-PDE-lower-bound-from-BAS}
Fix a bicharacteristic segment
$(x(t),\xi(t),a(t))$, $0\le t\le T$, contained in the interior
action--angle band, with $\xi(0)\ne0$, $a(0)\ne0$, and
$a(0)\cdot\xi(0)=0$.  Then
\begin{equation}
  \|G(T)\|_{\rm ess}
  \ge \frac{|a(T)|}{|a(0)|}.
  \label{eq:essential-norm-BAS-lower-bound}
\end{equation}
\end{proposition}

\begin{proof}
For simplification, we denote
\begin{equation*}
  x_0:=x(0),\qquad \xi_0:=\xi(0),\qquad a_0:=a(0).
\end{equation*}
Let $\mathfrak r>0$ be a spatial localization scale and $\varepsilon>0$ be a wavelength parameter.
Choose $h\in C_0^\infty(\R^3)$ with $\|h\|_{L^2}=1$ and define
\begin{equation*}
  h_{\mathfrak r}(x)
  :=\mathfrak r^{-3/2}
  h\!\left(\frac{x-x_0}{\mathfrak r}\right).
\end{equation*}
For all sufficiently small fixed $\mathfrak r$, the support of
$h_{\mathfrak r}$ lies strictly inside the action--angle band.
We also define
\begin{equation}
  f_{\varepsilon,\mathfrak r}
  :=\PP\left[
    h_{\mathfrak r}(x)a_0
    e^{i\xi_0\cdot(x-x_0)/\varepsilon}
  \right],
  \label{eq:two-parameter-div-free-packet}
\end{equation}
where $\PP$ is the Leray projection. Since $a_0\cdot\xi_0=0$, it follows from a direct Fourier-multiplier calculation that for
each fixed $\mathfrak r$,
\begin{equation}
  f_{\varepsilon,\mathfrak r}
  =h_{\mathfrak r}a_0
   e^{i\xi_0\cdot(x-x_0)/\varepsilon}+o_{L^2}(1),
  \qquad
  \|f_{\varepsilon,\mathfrak r}\|_{L^2}
  \longrightarrow |a_0|.
  \label{eq:Leray-fixed-radius-asymptotic}
\end{equation}

Applying Lemma~\ref{lem:fixed-time-Euler-geometric-optics-transfer} with
$s=0$, $j=\varepsilon^{-1}$,
$\phi(x)=\xi_0\cdot(x-x_0)$, and
$A(x)=h_{\mathfrak r}(x)a_0$, we obtain for fixed $\mathfrak r$ and
$\varepsilon\downarrow0$ that
\begin{equation}
  G(T)f_{\varepsilon,\mathfrak r}(X_T(x))
  =e^{i\xi_0\cdot(x-x_0)/\varepsilon}
   h_{\mathfrak r}(x)
   \mathscr U(T;x,\xi_0)a_0+o_{L^2}(1),
  \label{eq:fixed-radius-short-wave-expansion}
\end{equation}
where the error may depend on $\mathfrak r$ and $T$. Changing variables from the initial point $x$ to the
evolved point $X_T(x)$ in \eqref{eq:fixed-radius-short-wave-expansion} and employing $\det DX_T=1$ and
\eqref{eq:Leray-fixed-radius-asymptotic} yields
\begin{equation}
  \lim_{\varepsilon\downarrow0}
  \frac{\|G(T)f_{\varepsilon,\mathfrak r}\|_{L^2}^2}
       {\|f_{\varepsilon,\mathfrak r}\|_{L^2}^2}
  =R_{\mathfrak r}(T)^2
  \label{eq:fixed-radius-amplification-limit}
\end{equation}
with
\begin{equation*}
  R_{\mathfrak r}(T)^2
  :=\frac{\displaystyle\int_{\R^3}
    |\mathscr U(T;x,\xi_0)a_0|^2
    |h_{\mathfrak r}(x)|^2\,\dd x}
  {|a_0|^2\displaystyle\int_{\R^3}
    |h_{\mathfrak r}(x)|^2\,\dd x}.
\end{equation*}
Due to continuity of the cocycle on its initial point,  we obtain
\begin{equation}
  \lim_{\mathfrak r\downarrow0}
  R_{\mathfrak r}(T)
  =\frac{|\mathscr U(T;x_0,\xi_0)a_0|}{|a_0|}
  =\frac{|a(T)|}{|a(0)|}.
  \label{eq:localization-radius-limit}
\end{equation}

For each fixed
$\mathfrak r$, denote
\begin{equation*}
  \widehat f_{\varepsilon,\mathfrak r}
  :=\frac{f_{\varepsilon,\mathfrak r}}
  {\|f_{\varepsilon,\mathfrak r}\|_{L^2}}.
\end{equation*}
Since $\widehat f_{\varepsilon,\mathfrak r}\rightharpoonup0$ in
$L^2_{\rm div}$ as $\varepsilon\downarrow0$, we have $\mathcal C\widehat f_{\varepsilon,\mathfrak r}\to0$ strongly for any compact operator $\mathcal C$.
Equation
\eqref{eq:fixed-radius-amplification-limit} and the triangle inequality then imply
\begin{equation*}
  \|G(T)-\mathcal C\|
  \ge\limsup_{\varepsilon\downarrow0}
  \|(G(T)-\mathcal C)
    \widehat f_{\varepsilon,\mathfrak r}\|_{L^2}
  \ge R_{\mathfrak r}(T).
\end{equation*}
We obtain \eqref{eq:essential-norm-BAS-lower-bound} by taking first the infimum over compact $\mathcal C$ and then the limit
$\mathfrak r\downarrow0$.
\end{proof}


\begin{proposition}\label{prop:axisymmetric-fixed-time-realization}
Let an axisymmetric bicharacteristic be generated in the action--angle chart
by the initial phase
\begin{equation*}
  S_0(y):=S(0,y)=n\sigma+\ell I,
  \qquad y=(\sigma,\beta,I),
  \qquad n\in\mathbb Z\setminus\{0\},\quad \ell\in\mathbb R,
\end{equation*}
where the initial torus $I=I_0$ lies in the interior action--angle band.
Let $(x(t),\xi(t),a(t))$, $0\le t\le T$, be one of its meridional
representatives with $a(0)\ne0$ and $a(0)\cdot\xi(0)=0$.  Then
\begin{equation}
  \left\|G(T)|_{H^s_{\rm axi}}\right\|_{{\rm ess},H^s}
  \ge
  \frac{|\xi(T)|^s|a(T)|}{|\xi(0)|^s|a(0)|},
  \qquad s\ge0.
  \label{eq:axisymmetric-essential-lower-bound}
\end{equation}
Moreover, for every $\eta>0$ there is a smooth, real-valued,
axisymmetric divergence-free field $v_0$ with $\|v_0\|_{H^s}=1$ such that
\begin{equation*}
  \|G(T)v_0\|_{H^s}
  \ge
  \frac{|\xi(T)|^s|a(T)|}{|\xi(0)|^s|a(0)|}-\eta.
\end{equation*}
\end{proposition}

\begin{proof}
Choose an interval on $\mathbb T_\sigma=\mathbb R/(2\pi\mathbb Z)$ containing the
initial meridional angle $\sigma_0$ on which the quotient map is injective.
Consider $S_0=n\sigma+\ell I$ on this interval. Denote
\[  y_0=(\sigma_0,\beta_0,I_0), \qquad x(0)=\Phi(y_0), \]
and the physical initial covector field generated by the coordinate
phase $S_0$ by
\begin{equation*}
  \xi_{S_0}(y)
  :=\mathsf E(y)^{-\mathsf T}\dd_yS_0(y),
  \qquad y=(\sigma,\beta,I).
\end{equation*}
Thus $\xi_{S_0}(y_0)=\xi(0)$.  Choose
$h^{\rm axi}\in C_0^\infty(\R^2)$ with
$\|h^{\rm axi}\|_{L^2(\R^2)}=1$ and define
\begin{equation*}
  h_{\mathfrak r}^{\rm axi}(\sigma,I)
  :=\mathfrak r^{-1}
  h^{\rm axi}\!\left(
    \frac{\sigma-\sigma_0}{\mathfrak r},
    \frac{I-I_0}{\mathfrak r}
  \right).
\end{equation*}
For all sufficiently small $\mathfrak r>0$, $ h_{\mathfrak r}^{\rm axi}$ is
supported in the chosen interval and strictly inside the action band. We extend
$a(0)$ arbitrarily to a smooth
axisymmetric physical field $b(y)$ on this support with $b(y_0)=a(0)$ and define
\begin{equation}
  \widetilde a_0(y)
  :=b(y)-\frac{b(y)\cdot\xi_{S_0}(y)}{|\xi_{S_0}(y)|^2}
    \xi_{S_0}(y).
  \label{eq:axisymmetric-explicit-transverse-extension}
\end{equation}
Since $n\ne0$ and $\mathsf E(y)$ is invertible,
$\xi_{S_0}(y)\ne0$ throughout the chart.  Thus $\widetilde a_0$ is a smooth
axisymmetric physical amplitude satisfying
\begin{equation}
  \widetilde a_0(y_0)=a(0),
  \qquad
  \widetilde a_0(y)\cdot\xi_{S_0}(y)=0.
  \label{eq:axisymmetric-amplitude-extension}
\end{equation}
For integers $j\to\infty$, define
\begin{equation}
  f_{j,\mathfrak r}^{\rm axi}:=\PP\left[
    h_{\mathfrak r}^{\rm axi}(\sigma,I)
    \widetilde a_0(y)e^{ijS_0(y)}
  \right].
  \label{eq:axisymmetric-quantized-packets}
\end{equation}
One can verify that $e^{ijn\sigma}$ is single valued on the meridional circle due to the choice of an integer $j$.  We note $f_{j,\mathfrak r}^{\rm axi}$ is axisymmetric.

In view of \eqref{eq:LH-xi}-\eqref{eq:LH-a}, we have
\begin{equation}\label{U-homo}
  \mathscr U(T;x,j\xi)=\mathscr U(T;x,\xi), \qquad j>0.
\end{equation}
The principal cocycle of the conjugated evolution
$\langle D\rangle^sG(T)\langle D\rangle^{-s}$ is
\begin{equation}\label{Us-cocycle}
  \mathscr U_s(T;x,\zeta)
  =\left( \frac{|DX_T(x)^{-\mathsf T}\zeta|}{|\zeta|} \right)^s \mathscr U(T;x,\zeta).
\end{equation}

Define the
initial phase $\phi_0$ on the packet support and its transported phase by
\begin{equation*}
  \phi_0(\Phi(y)):=S_0(y),
  \qquad
  \phi_T(x):=\phi_0(X_{-T}(x)),
\end{equation*}
and the transported physical covector by
\begin{equation*}
  \xi_{S_0,T}(y)
  :=DX_T(\Phi(y))^{-\mathsf T}\xi_{S_0}(y).
\end{equation*}
We deduce 
\begin{equation*}
  \dd\phi_T(X_T(\Phi(y)))
  =DX_T(\Phi(y))^{-\mathsf T}
    \mathsf E(y)^{-\mathsf T}\dd_yS_0(y)
  =\xi_{S_0,T}(y).
\end{equation*}
The transversality in \eqref{eq:axisymmetric-amplitude-extension} verifies
the hypothesis of Lemma~\ref{lem:fixed-time-Euler-geometric-optics-transfer}.  So we apply Lemma~\ref{lem:fixed-time-Euler-geometric-optics-transfer} with the physical amplitude defined on the packet support by
\begin{equation*}
  A(\Phi(y))
  :=h_{\mathfrak r}^{\rm axi}(\sigma,I)\widetilde a_0(y).
\end{equation*}
It follows from \eqref{U-homo} and \eqref{Us-cocycle} that
\begin{align*}
 &\mathscr U_s(T;\Phi(y),j\xi_{S_0}(y))
       \bigl[|\xi_{S_0}(y)|^s\widetilde a_0(y)\bigr]\\
 &\quad=
   \left(\frac{|j\xi_{S_0,T}(y)|}{|j\xi_{S_0}(y)|}\right)^s
   |\xi_{S_0}(y)|^s
   \mathscr U(T;\Phi(y),j\xi_{S_0}(y))\widetilde a_0(y)\\
 &\quad=
   |\xi_{S_0,T}(y)|^s
   \mathscr U(T;\Phi(y),\xi_{S_0}(y))\widetilde a_0(y).
\end{align*}
For fixed $\mathfrak r$, $T$, and $s$, Lemma~\ref{lem:fixed-time-Euler-geometric-optics-transfer} implies
\begin{align}
  j^{-s}\bigl(\langle D\rangle^s
    f_{j,\mathfrak r}^{\rm axi}\bigr)(\Phi(y))
  &=e^{ijS_0(y)}|\xi_{S_0}(y)|^s
    h_{\mathfrak r}^{\rm axi}(\sigma,I)\widetilde a_0(y)
    +o_{L^2}(1),
    \label{eq:axisymmetric-Hs-WKB-input}\\
  j^{-s}\bigl(\langle D\rangle^sG(T)
    f_{j,\mathfrak r}^{\rm axi}\bigr)(X_T(\Phi(y)))
  &=e^{ijS_0(y)}|\xi_{S_0,T}(y)|^s
    h_{\mathfrak r}^{\rm axi}(\sigma,I)
    \mathscr U(T;\Phi(y),\xi_{S_0}(y))\widetilde a_0(y)
    +o_{L^2}(1).
    \label{eq:axisymmetric-Hs-WKB-output}
\end{align}
Here and below, the error terms $o_{L^2}(1)$ are understood as $j\to\infty$ with
$\mathfrak r$ and $T$ fixed. 

Applying the fact $\det D\Phi=\det DX_T=1$ and hence $\dd x=\dd\sigma\,\dd\beta\,\dd I$, we obtain
\begin{equation}
  \lim_{j\to\infty}
  \frac{\|G(T)f_{j,\mathfrak r}^{\rm axi}\|_{H^s}^2}
       {\|f_{j,\mathfrak r}^{\rm axi}\|_{H^s}^2}
  =R_{\mathfrak r,s}^{\rm axi}(T)^2
  \label{eq:axisymmetric-amplification-limit}
\end{equation}
with
\begin{equation*}
  \bigl(R_{\mathfrak r,s}^{\rm axi}(T)\bigr)^2
  :=\frac{\displaystyle\int
    |\xi_{S_0,T}(y)|^{2s}
    |\mathscr U(T;\Phi(y),\xi_{S_0}(y))
      \widetilde a_0(y)|^2
    |h_{\mathfrak r}^{\rm axi}(\sigma,I)|^2
    \,\dd\sigma\,\dd I}
  {\displaystyle\int
    |\xi_{S_0}(y)|^{2s}|\widetilde a_0(y)|^2
    |h_{\mathfrak r}^{\rm axi}(\sigma,I)|^2
    \,\dd\sigma\,\dd I}, \quad y=(\sigma,\beta_0,I).
\end{equation*}
Continuity of the finite-time flow and cocycle implies
\begin{equation*}
  \lim_{\mathfrak r\downarrow0}
  R_{\mathfrak r,s}^{\rm axi}(T)
  =\frac{|\xi(T)|^s|a(T)|}{|\xi(0)|^s|a(0)|}.
\end{equation*}

For fixed $\mathfrak r>0$, it follows 
\begin{equation*}
  \widehat f_{j,\mathfrak r}^{\rm axi}
  :=
  \frac{f_{j,\mathfrak r}^{\rm axi}}
       {\|f_{j,\mathfrak r}^{\rm axi}\|_{H^s}}   \rightharpoonup0
  \quad\text{in }H^s_{\rm axi}
  \qquad\text{as }j\to\infty,
\end{equation*}
and consequently, 
\begin{equation*}
  \|\mathcal C\widehat f_{j,\mathfrak r}^{\rm axi}\|_{H^s}
  \longrightarrow0, \qquad \mbox{for every compact operator} \quad \mathcal C:H^s_{\rm axi}\to H^s_{\rm axi}.
\end{equation*}
Then combined with \eqref{eq:axisymmetric-amplification-limit} we infer
\begin{align*}
  \left\|G(T)|_{H^s_{\rm axi}}-\mathcal C\right\|_{H^s\to H^s}
  &\ge
  \limsup_{j\to\infty}
  \left\|
    (G(T)-\mathcal C)
    \widehat f_{j,\mathfrak r}^{\rm axi}
  \right\|_{H^s} 
  \ge R_{\mathfrak r,s}^{\rm axi}(T).
\end{align*}
Taking the infimum over compact $\mathcal C$ and then letting
$\mathfrak r\to 0$ yields
\begin{equation*}
  \left\|G(T)|_{H^s_{\rm axi}}\right\|_{{\rm ess},H^s}
  \ge
  \lim_{\mathfrak r\downarrow0}
  R_{\mathfrak r,s}^{\rm axi}(T)
  =
  \frac{|\xi(T)|^s|a(T)|}
       {|\xi(0)|^s|a(0)|}
\end{equation*}
which implies \eqref{eq:axisymmetric-essential-lower-bound} over the complex space.

It remains to show the existence of a real testing sequence satisfying this lower
bound.  Assume $a(0)=u_0+iw_0$.  Since $\xi(0)$ is real, both $u_0$ and $w_0$
belong to $\xi(0)^\perp$, and we have
\begin{align*}
  |a(0)|^2&=|u_0|^2+|w_0|^2,\notag\\
  |a(T)|^2
  &=|\mathscr U(T;x(0),\xi(0))u_0|^2
    +|\mathscr U(T;x(0),\xi(0))w_0|^2.
\end{align*}
Therefore one of the nonzero vectors $u_0,w_0$, denoted by $b_0$, satisfies
\begin{equation}
  \frac{|\mathscr U(T;x(0),\xi(0))b_0|}{|b_0|}
  \ge \frac{|a(T)|}{|a(0)|}.
  \label{eq:real-amplitude-component-choice}
\end{equation}
We apply \eqref{eq:axisymmetric-explicit-transverse-extension} to extend $b_0$ to a
smooth real axisymmetric amplitude $\widetilde b_0(y)$ perpendicular to
$\xi_{S_0}(y)$, and define
\begin{equation}
  g_{j,\mathfrak r}^{\rm axi}
  :=\PP\left[
    h_{\mathfrak r}^{\rm axi}(\sigma,I)\widetilde b_0(y)
    \cos(jS_0(y))
  \right],
  \label{eq:axisymmetric-real-cosine-packets}
\end{equation}
which are smooth, real, axisymmetric, and divergence free.  Define the two complex WKB branches
\begin{equation*}
  g_{j,\mathfrak r}^{\rm axi,\pm}
  :=
  \PP\left[
    h_{\mathfrak r}^{\rm axi}(\sigma,I)
    \widetilde b_0(y)e^{\pm ijS_0(y)}
  \right].
\end{equation*}
Then
\begin{equation*}
  g_{j,\mathfrak r}^{\rm axi}
  =
  \frac12\left(
    g_{j,\mathfrak r}^{\rm axi,+}
    +g_{j,\mathfrak r}^{\rm axi,-}
  \right).
\end{equation*}
Since the bicharacteristic--amplitude system is invariant under
$\xi\mapsto-\xi$, we have
\begin{equation*}
  \mathscr U(T;x,-\xi)=\mathscr U(T;x,\xi),
\end{equation*}
and hence the two branches have the same amplification.  Moreover, it is clear that
\begin{equation*}
  \dd\phi_0(\Phi(y))=\xi_{S_0}(y)\ne0,
  \qquad
  \dd\phi_T(X_T(\Phi(y)))=\xi_{S_0,T}(y)\ne0,
\end{equation*}
since $X_T$ is a diffeomorphism.
One can verify that
\begin{align*}
  \left\langle
    g_{j,\mathfrak r}^{\rm axi,+},
    g_{j,\mathfrak r}^{\rm axi,-}
  \right\rangle_{H^s}
  &=o(j^{2s}),\\
  \left\langle
    G(T)g_{j,\mathfrak r}^{\rm axi,+},
    G(T)g_{j,\mathfrak r}^{\rm axi,-}
  \right\rangle_{H^s}
  &=o(j^{2s})
\end{align*}
as $j\to\infty$, with $\mathfrak r$ and $T$ fixed.  Moreover,
$g_{j,\mathfrak r}^{\rm axi,-}
=\overline{g_{j,\mathfrak r}^{\rm axi,+}}$, and the reality of $G(T)$ implies
\begin{equation*}
  G(T)g_{j,\mathfrak r}^{\rm axi,-}
  =
  \overline{G(T)g_{j,\mathfrak r}^{\rm axi,+}}.
\end{equation*}
It then follows
\begin{align*}
  \|g_{j,\mathfrak r}^{\rm axi}\|_{H^s}^2
  &=
  \frac12
  \|g_{j,\mathfrak r}^{\rm axi,+}\|_{H^s}^2
  +o(j^{2s}),\\
  \|G(T)g_{j,\mathfrak r}^{\rm axi}\|_{H^s}^2
  &=
  \frac12
  \|G(T)g_{j,\mathfrak r}^{\rm axi,+}\|_{H^s}^2
  +o(j^{2s})
\end{align*}
and hence by \eqref{eq:real-amplitude-component-choice}
\begin{align*}
  &\lim_{\mathfrak r\downarrow0}\lim_{j\to\infty}
  \frac{\|G(T)g_{j,\mathfrak r}^{\rm axi}\|_{H^s}^2}
       {\|g_{j,\mathfrak r}^{\rm axi}\|_{H^s}^2}=
  \frac{|\xi(T)|^{2s}
    |\mathscr U(T;x(0),\xi(0))b_0|^2}
  {|\xi(0)|^{2s}|b_0|^2}
  \ge
  \frac{|\xi(T)|^{2s}|a(T)|^2}
       {|\xi(0)|^{2s}|a(0)|^2}.
\end{align*}

Finally, we have
\begin{equation*}
  \widehat g_{j,\mathfrak r}^{\rm axi}
  :=
  \frac{g_{j,\mathfrak r}^{\rm axi}}
       {\|g_{j,\mathfrak r}^{\rm axi}\|_{H^s}}   \rightharpoonup0
  \quad\text{in the real axisymmetric subspace},
\end{equation*}
and hence
\begin{equation*}
  \|\mathcal C\widehat g_{j,\mathfrak r}^{\rm axi}\|_{H^s}
  \longrightarrow0, \qquad \mbox{for every compact operator} \quad \mathcal C:H^s_{\rm axi}\to H^s_{\rm axi},
\end{equation*}
which implies \eqref{eq:axisymmetric-essential-lower-bound} over the real space by taking first
$j\to\infty$ and then $\mathfrak r\to 0$. Given $\eta>0$, we may also choose first $\mathfrak r$ sufficiently small and
then $j$ sufficiently large so that $v_0:=\widehat g_{j,\mathfrak r}^{\rm axi}$ satisfies the final assertion of the proposition.
\end{proof}

\begin{proof}[Proof of
Theorem~\ref{thm:critical-torus-exponential-essential-instability}]
Proposition~\ref{prop:critical-torus-exponential-LH} shows that there exists an axisymmetric
ray whose reduced cylindrical cocycle is periodic and has Floquet rate
$\sigma_\ell>0$, with
\begin{equation*}
  |a(t)|\ge c_\ell e^{\sigma_\ell t}|a(0)|,
  \qquad t\ge0.
\end{equation*}
It follows from Proposition~\ref{prop:axisymmetric-fixed-time-realization} that, for fixed $s\ge0$ and every $t\ge0$,
\begin{equation}
  \left\|G_\iota(t)|_{H^s_{\rm axi}}\right\|_{{\rm ess},H^s}
  \ge
  c_\ell e^{\sigma_\ell t}
  \frac{|\xi(t)|^s}{|\xi(0)|^s}.
  \label{eq:critical-torus-essential-bound-with-covector}
\end{equation}
At $I=I_{\rm c}$, we have $q(t)=(n,0,\ell)$ from \eqref{eq:wave-covector}, while
$\sigma(t)=\sigma_0+\Omega_1(I_{\rm c})t$.  Since the cylindrical
representation of $\mathsf E(\sigma,\beta,I_{\rm c})$ is $2\pi$-periodic
in $\sigma$, the physical covector norm satisfies
\begin{equation*}
  |\xi(t+T_{0,c})|=|\xi(t)|.
\end{equation*}
Moreover, we have
\begin{equation*}
  \xi(t)=DX_t(x(0))^{-\mathsf T}\xi(0)\ne0,
\end{equation*}
since $X_t$ is a diffeomorphism and $\xi(0)\ne0$.  Continuity and periodicity therefore imply
\begin{equation*}
  \min_{0\le\tau\le T_{0,c}}|\xi(\tau)|>0.
\end{equation*}
Hence \eqref{eq:critical-torus-essential-bound-with-covector} implies
\begin{equation*}
  \left\|G_\iota(t)|_{H^s_{\rm axi}}\right\|_{{\rm ess},H^s}
  \ge c_{\ell,s}e^{\sigma_\ell t}
\end{equation*}
with $c_{\ell,s}:=c_\ell\min_{0\le\tau\le T_{0,c}}\frac{|\xi(\tau)|^s}{|\xi(0)|^s}>0$.
This concludes \eqref{eq:critical-torus-exact-exponential-essential-growth}.

For each fixed observation time $T$, we apply the final assertion of
Proposition~\ref{prop:axisymmetric-fixed-time-realization} with $\eta=\frac12c_{\ell,s}e^{\sigma_\ell T}$
and thus obtain a smooth real axisymmetric divergence-free field
$v_{0,T}$ with $\|v_{0,T}\|_{H^s}=1$ and
\begin{align}
  \|G_\iota(T)v_{0,T}\|_{H^s}
  &\ge
  \frac{|\xi(T)|^s|a(T)|}{|\xi(0)|^s|a(0)|}-\frac12c_{\ell,s}e^{\sigma_\ell T}\notag\\
  &\ge c_{\ell,s}e^{\sigma_\ell T}-\frac12c_{\ell,s}e^{\sigma_\ell T}
   =\frac12c_{\ell,s}e^{\sigma_\ell T}.
  \label{eq:critical-torus-real-packet-with-prefactor}
\end{align}
In the proof of Proposition~\ref{prop:axisymmetric-fixed-time-realization}, one first chooses the
meridional localization scale $\mathfrak r=\mathfrak r(T)$ sufficiently
small and then the integer $j=j(T,\mathfrak r)$
sufficiently large. For every $\delta>0$, one has
$\frac12c_{\ell,s}e^{\delta T}\ge1$ for all sufficiently large $T$. Therefore \eqref{eq:critical-torus-observation-time-real-packet} follows from \eqref{eq:critical-torus-real-packet-with-prefactor} for $0<\delta<\sigma_\ell$.
\end{proof}


\subsection{Growth from fixed initial data}

\begin{proof}[Proof of Corollary~\ref{cor:generic-fixed-Sobolev-datum}]
Throughout this proof, all operator norms are taken on $H^s_{\rm div}(\R^3;\R^3)$ unless otherwise
specified.
Fix a rational constant $\kappa$ with
$0<\kappa<\sigma_\ell$, and choose
\begin{equation*}
  0<\delta_\kappa<\sigma_\ell-\kappa.
\end{equation*}
Theorem~\ref{thm:critical-torus-exponential-essential-instability} and the
real-valued packets in
\eqref{eq:critical-torus-observation-time-real-packet} imply that at every
sufficiently large integer time $T$,
\begin{equation*}
  e^{-\kappa T}\|G_\iota(T)\|_{H^s\to H^s}
  \ge e^{(\sigma_\ell-\delta_\kappa-\kappa)T}.
\end{equation*}
Since $\sigma_\ell-\delta_\kappa-\kappa>0$, it follows that
\begin{equation}
  \sup_{\substack{T\in\mathbb N\\T\ge2}}
  \bigl\|e^{-\kappa T}G_\iota(T)\bigr\|_{H^s\to H^s}=\infty.
  \label{eq:weighted-propagators-not-uniformly-bounded}
\end{equation}

For $N\in\mathbb N$, define
\begin{equation*}
  E_{\kappa,N}
  :=\left\{v\in H^s_{\rm div}(\R^3;\R^3):
  \sup_{\substack{T\in\mathbb N\\T\ge2}}
  e^{-\kappa T}\|G_\iota(T)v\|_{H^s}\le N\right\}.
\end{equation*}
We note each $E_{\kappa,N}$ is closed. Indeed, it is the intersection, over the countable set of
integer times $T\ge2$, of the inverse images of the closed interval $[0,N]$
under the continuous maps
$v\mapsto e^{-\kappa T}\|G_\iota(T)v\|_{H^s}$.  We also claim $E_{\kappa,N}$ has empty interior.
Indeed, if it has non-empty interior, i.e.
$B(v_*,r)\subset E_{\kappa,N}$ for some $r>0$, then for every
$h\in H^s_{\rm div}$ with $\|h\|_{H^s}<r$, we have
\begin{equation*}
  e^{-\kappa T}\|G_\iota(T)h\|_{H^s}
  \le 2N,
  \qquad T\in\mathbb N,\quad T\ge2.
\end{equation*}
Choosing $h=(r/2)w$ and $\|w\|_{H^s}=1$, and then taking the
supremum first over $w$ and then over $T$ yields
\begin{equation*}
  \sup_{\substack{T\in\mathbb N\\T\ge2}}
  \bigl\|e^{-\kappa T}G_\iota(T)\bigr\|_{H^s\to H^s}
  \le \frac{4N}{r},
\end{equation*}
which contradicts \eqref{eq:weighted-propagators-not-uniformly-bounded}.  Thus
$E_{\kappa,N}$ is
nowhere dense, and
\begin{equation*}
  \mathcal R_{s,\kappa}
  :=H^s_{\rm div}(\R^3;\R^3)
  \setminus\bigcup_{N=1}^\infty E_{\kappa,N}
\end{equation*}
is residual.  Thus we have
\begin{equation}
  \sup_{\substack{T\in\mathbb N\\T\ge2}}
  e^{-\kappa T}\|G_\iota(T)v\|_{H^s}=\infty, \quad v\in\mathcal R_{s,\kappa}.
  \label{eq:fixed-datum-growth-below-rate}
\end{equation}

Now define
\begin{equation*}
  \mathcal R_s
  :=\bigcap_{\substack{\kappa\in\mathbb Q\\0<\kappa<\sigma_\ell}}
  \mathcal R_{s,\kappa}
\end{equation*}
which is residual as well.  Suppose that \eqref{eq:fixed-datum-subsequential-growth} is false
for some $v_0\in\mathcal R_s$. Then there is a rational number $\kappa$ such that
\begin{equation*}
    \limsup_{t\to\infty} \frac1t\log\bigl(1+\|G_\iota(t)v_0\|_{H^s}\bigr)<\kappa<\sigma_\ell.
\end{equation*}
It follows that for every sufficiently large integer $T$
\begin{equation*}
  \frac1T\log\bigl(1+\|G_\iota(T)v_0\|_{H^s}\bigr)<\kappa,
\end{equation*}
and hence
$e^{-\kappa T}\|G_\iota(T)v_0\|_{H^s}\le1$.  This contradicts \eqref{eq:fixed-datum-growth-below-rate}.  Hence we complete the proof of \eqref{eq:fixed-datum-subsequential-growth}.
Since $\sigma_\ell>0$, \eqref{eq:fixed-datum-subsequential-growth} also implies
$\sup_{t\ge0}\|G_\iota(t)v_0\|_{H^s}=\infty$.

Finally, we can repeat the same argument in the closed invariant real axisymmetric
subspace.  The real axisymmetric packets in
\eqref{eq:critical-torus-observation-time-real-packet} establish
\eqref{eq:weighted-propagators-not-uniformly-bounded} for the restricted
operator family, so the Baire argument above applies as well.
\end{proof}

\appendix

\section{Baldi formulas used in the small-action expansions}
\label{app:Baldi-formulas}

This appendix collects the formulas
from Baldi's construction that are used in the proofs of
Propositions~\ref{prop:first-geometric-expansion},~\ref{prop:negative-averaged-determinant},
and~\ref{prop:critical-torus-negative-determinant} in the notation of the present paper.
Only the identities and coefficients needed for these small-action
expansions are recorded; their construction and analytic justification are
contained in Sections~3.5--3.10 and 4.1--4.8 of \cite{Baldi2024}.  The
remainders below are uniform in the angular variables.  Baldi's analyticity
also allows differentiation and integration to every fixed order.

\textbf{Basic parameters and correction functions.}
Baldi denotes his unscaled level parameter by $c$ which has the form $c=h(I)$
after introducing the action variable.  With the pressure normalization used
in this paper, this parameter is denoted
\begin{equation*}
  \mathfrak c:=4\Kcal(I),
  \qquad \mbox{with} \quad
  \mu:=\sqrt{\mathfrak c},
  \qquad
  r:=\sqrt{2I}.
\end{equation*}
Baldi uses $\gamma_c(\vartheta)$ for the meridional level graph.  To keep it
distinct here from the azimuthal correction $\gamma(\sigma,I)$ in
\eqref{eq:Phi-explicit}, we denote the former by
$\gamma^{\rm mer}_{\mathfrak c}(\vartheta)$.  Baldi denotes the latter by
$\eta(\sigma,I)$.  Thus
\begin{equation*}
  \gamma^{\rm mer}_{\mathfrak c}(\vartheta)
  =\gamma_c(\vartheta)\ \text{in Baldi's notation},
  \qquad
  \gamma(\sigma,I)=\eta(\sigma,I)\ \text{in Baldi's notation}.
\end{equation*}

\textbf{The meridional level graph.}
Baldi introduces an analytic function $w(\vartheta,\mu)$ through
\begin{equation}
  \gamma^{\rm mer}_{\mathfrak c}(\vartheta)
  =\frac12\mu^2w(\vartheta,\mu)^2,
  \qquad
  \sqrt{2\gamma^{\rm mer}_{\mathfrak c}(\vartheta)}
  =\mu w(\vartheta,\mu).
  \label{eq:Baldi-appendix-level-graph}
\end{equation}
The coefficients needed below are
\begin{align*}
  P_3(\vartheta)
    &=2\sin\vartheta-\sin(3\vartheta),
  \\
  w(\vartheta,\mu)
    &=W_0+W_1(\vartheta)\mu+O(\mu^2),
  \qquad
  W_0=\frac1{\sqrt2},
  \qquad
  W_1(\vartheta)=-\frac18P_3(\vartheta).
\end{align*}
Here the formula for $P_3$ is Baldi's (4.19), while the level graph and the
coefficients of $w$ follow from (4.26), (4.28), (4.36)--(4.37), and (4.43).
Baldi's symmetry formula (4.27) states that
\begin{equation*}
  w(-\vartheta,-\mu)=w(\vartheta,\mu).
\end{equation*}
Consequently,
$\mathcal Q_j^{\rm mer}(-\vartheta)
=(-1)^j\mathcal Q_j^{\rm mer}(\vartheta)$; in particular, every odd
coefficient has zero angular average.
More precisely, the convergent level expansion has the form
\begin{equation}
  \begin{gathered}
    \gamma^{\rm mer}_{\mathfrak c}(\vartheta)
      =\sum_{j=2}^{\infty}
        \mathcal Q_j^{\rm mer}(\vartheta)\mu^j,
    \\
    \mathcal Q_2^{\rm mer}=\frac14,
    \qquad
    \mathcal Q_3^{\rm mer}(\vartheta)
      =-\frac{P_3(\vartheta)}{8\sqrt2},
    \qquad
    \frac1{2\pi}\int_0^{2\pi}
      \mathcal Q_4^{\rm mer}(\vartheta)\,\dd\vartheta=0.
  \end{gathered}
  \label{eq:Baldi-appendix-level-coefficients}
\end{equation}
These identities correspond to Baldi's (4.39) and (4.42)--(4.44).  Since
$\mathfrak c=\mu^2$, it follows
\begin{align}
  \partial_{\mathfrak c}
    \gamma^{\rm mer}_{\mathfrak c}(\vartheta)
  &=\frac14+\frac32\mathcal Q_3^{\rm mer}(\vartheta)\mu
    +2\mathcal Q_4^{\rm mer}(\vartheta)\mu^2+O(\mu^3)
  \notag\\
  &=\frac14-\frac{3\mu}{16\sqrt2}P_3(\vartheta)+O(\mu^2).
  \label{eq:Baldi-appendix-level-derivative}
\end{align}

\textbf{Straightening the meridional angle.}
Following Baldi's formulas (3.42), (3.51), and (3.52), define
\begin{equation}
  F_{\mathfrak c}(\vartheta)
  :=\int_0^\vartheta
     \partial_{\mathfrak c}
       \gamma^{\rm mer}_{\mathfrak c}(u)\,\dd u,
  \qquad
  f_{\mathfrak c}(\vartheta)
  :=\frac{2\pi F_{\mathfrak c}(\vartheta)}
          {F_{\mathfrak c}(2\pi)},
  \qquad
  g_{\mathfrak c}:=f_{\mathfrak c}^{-1}.
  \label{eq:Baldi-appendix-angle-straightening}
\end{equation}
The old meridional angle and the straightened angle are related by
\begin{equation*}
  \vartheta=g_{\mathfrak c}(\sigma),
  \qquad
  \sigma=f_{\mathfrak c}(\vartheta).
\end{equation*}
The higher-order coefficients in Baldi's construction yields the sharper
expansion of the unscaled meridional period factor
\begin{equation}
  \mathcal T(\mathfrak c):=F_{\mathfrak c}(2\pi)
  =\frac\pi2\left(
    1-\frac{3195}{16384}\mathfrak c^2+O(\mathfrak c^3)
  \right),
  \label{eq:Baldi-appendix-T-expansion}
\end{equation}
see Baldi's formulas (4.48)--(4.50).  The normalization
of the action variable shows the action-period identity (see Baldi's
(3.57) and (3.60)--(3.61))
\begin{equation}
  \frac{\dd\mathfrak c}{\dd I}\,\mathcal T(\mathfrak c)=2\pi.
  \label{eq:Baldi-appendix-action-period-identity}
\end{equation}
Moreover,
\begin{equation*}
  \int_0^\vartheta P_3(u)\,\dd u
  =\frac53-2\cos\vartheta+\frac13\cos(3\vartheta),
\end{equation*}
and hence
\begin{align}
  f_{\mathfrak c}(\vartheta)
  &=\vartheta-
    \frac{\mu}{4\sqrt2}
    \bigl(5-6\cos\vartheta+\cos(3\vartheta)\bigr)
    +O(\mu^2),
  \notag\\
  g_{\mathfrak c}(\sigma)
  &=\sigma+
    \frac{\mu}{4\sqrt2}
    \bigl(5-6\cos\sigma+\cos(3\sigma)\bigr)
    +O(\mu^2).
  \label{eq:Baldi-appendix-g-expansion}
\end{align}

It follows from the normalization of the action variable by Baldi's formula (4.58),
\begin{equation}
  \mathfrak c=4\Kcal(I)=4I+O(I^3),
  \qquad
  \frac{\dd\mathfrak c}{\dd I}=4+O(I^2),
  \label{eq:Baldi-appendix-action-normalization}
\end{equation}
 which combined with $r=\sqrt{2I}$ leads to
\begin{equation*}
  \mu=\sqrt2\,r+O(r^5).
\end{equation*}

\textbf{Reconstruction of the physical coordinates.}
On the level $\mathfrak c=4\Kcal(I)$, denote
\begin{equation*}
  \vartheta=g_{\mathfrak c}(\sigma),
  \qquad
  \upsilon=\gamma^{\rm mer}_{\mathfrak c}(\vartheta).
\end{equation*}
Baldi's physical reconstruction formulas are given by
\begin{equation}
  \rho(\sigma,I)
  =1+\sqrt{2\upsilon}\sin\vartheta,
  \qquad
  \zeta(\sigma,I)
  =\frac{\sqrt{2\upsilon}\cos\vartheta}{\rho(\sigma,I)}.
  \label{eq:Baldi-appendix-rho-zeta}
\end{equation}
The analytic function entering the azimuthal velocity has expansion
\begin{equation}
  \mathcal H(\mathfrak c)
  =4\mathfrak c-\frac{21}{2}\mathfrak c^2
   +\frac{39}{32}\mathfrak c^3+O(\mathfrak c^4),
  \label{eq:Baldi-appendix-H-expansion}
\end{equation}
which is formula (4.10) of \cite{Baldi2024}.  In the straightened variables,
Baldi's formulas (3.64), (3.71), and (3.72) become
\begin{align}
  Q(\sigma,I)
    &:=\frac{\sqrt{\mathcal H(\mathfrak c)}}
             {\rho(\sigma,I)^2},
  \notag\\
  Q_0(I)
    &:=\frac1{2\pi}\int_0^{2\pi}Q(\sigma,I)\,\dd\sigma,
  \qquad
  \widetilde Q(\sigma,I):=Q(\sigma,I)-Q_0(I),
  \notag\\
  \partial_\sigma\gamma(\sigma,I)
    &=\frac{\widetilde Q(\sigma,I)}
            {\dd\mathfrak c/\dd I}.
  \label{eq:Baldi-appendix-gamma-derivative}
\end{align}
Note $\widetilde Q$ is mean-zero in $\sigma$, so
\eqref{eq:Baldi-appendix-gamma-derivative} has a $2\pi$-periodic primitive.
Fixing the additive constant by $\gamma(0,I)=0$, we have
\begin{equation}
  \gamma(\sigma,I)
  =\left(\frac{\dd\mathfrak c}{\dd I}\right)^{-1}
    \int_0^\sigma
      \bigl(Q(\widetilde\sigma,I)-Q_0(I)\bigr)
      \,\dd\widetilde\sigma.
  \label{eq:Baldi-appendix-gamma-primitive}
\end{equation}
Together, \eqref{eq:Baldi-appendix-level-graph},
\eqref{eq:Baldi-appendix-g-expansion},
\eqref{eq:Baldi-appendix-action-normalization}, and
\eqref{eq:Baldi-appendix-rho-zeta}--
\eqref{eq:Baldi-appendix-gamma-primitive} are the external inputs
used to derive the first three expansions in
Proposition~\ref{prop:first-geometric-expansion}.  The action-period identity and the
sharp expansions \eqref{eq:Baldi-appendix-T-expansion} and
\eqref{eq:Baldi-appendix-H-expansion} are used in
Proposition~\ref{prop:negative-averaged-determinant}; those conclusions are then applied in
the cutoff argument in Proposition~\ref{prop:critical-torus-negative-determinant}.

\end{document}